\documentclass{amsart}
\usepackage{amsmath,amssymb,amsthm,amscd}
\usepackage{pgfplots}
\usepackage{mathrsfs}
\newtheorem{formula}{}[section]
\newtheorem{proposition}[formula]{Proposition}
\newtheorem{corollary}[formula]{Corollary}
\newtheorem{lemma}[formula]{Lemma}
\newtheorem{theorem}[formula]{Theorem}
\theoremstyle{definition}
\newtheorem{definition}[formula]{Definition}

\theoremstyle{remark}
\newtheorem*{remark}{Remark}

\begin{document}

\title{Graded Thread Modules over the Positive Part of the Witt (Virasoro) Algebra}
\author{Dmitry V. Millionschikov}
\thanks{Supported by the RFBR grant N 16-51-55017 (Russian Foundation for Basic Research)}
\address{Department of Mechanics and Mathematics, Moscow
State University, 119992 Moscow, RUSSIA}
\email{million@mech.math.msu.su}

\begin{abstract}
We study ${\mathbb Z}$-graded thread $W^+$-modules $$V=\oplus_i V_i, \; \dim{V_i}=1,  -\infty \le k< i < N\le +\infty, \; \dim{V_i}=0, \; {\rm \; otherwise},$$ 
over the positive part $W^+$ of the Witt (Virasoro) algebra $W$. There is well-known example  of infinite-dimensional ($k=-\infty, N=\infty$) two-parametric family $V_{\lambda, \mu}$ of $W^+$-modules induced by the twisted $W$-action on tensor densities $P(x)x^{\mu}(dx)^{-\lambda}, \mu, \lambda \in {\mathbb K}, P(x) \in {\mathbb K}[t]$. Another family $C_{\alpha, \beta}$ of $W^+$-modules is defined by the action of two multiplicative  generators $e_1, e_2$ of $W^+$ as $e_1f_i=\alpha f_{i+1}$ and $e_2f_j=\beta f_{j+2}$ for $i,j \in {\mathbb Z}$ and $\alpha, \beta$ are two arbitrary constants ($e_if_j=0, i \ge 3$). 

We classify $(n+1)$-dimensional graded thread $W^+$-modules for $n$ sufficiently large $n$ of three important types. New examples of graded thread $W^+$-modules different from finite-dimensional quotients of $V_{\lambda, \mu}$ and $C_{\alpha, \beta}$ were found.
\end{abstract}
\date{}

\maketitle
\section{Positive part of the Witt algebra and its finite-dimensional graded modules}

\label{sec:1}
The representation theory of the Virasoro algebra $Vir$ was intensively
studied in the 80s of the last century, now one may consider it as a classical section of the representation theory of infinite-dimensional Lie algebras. For instance, the structure of  Verma modules over Virasoro algebra and Fock
modules over $Vir$ were completely determined by B. Feigin and D. Fuchs.
O. Mathieu \cite{M} proved V. Kac's conjecture  which
says that any simple ${\mathbb Z}$-graded $Vir$-module with finite-dimensional homogeneous components is either a highest weight module, a lowest weight module, or the module of type $V_{\lambda, \mu}$.
However  two particular cases of
the theorem were already proved: the classification of Harish-Chandra modules for
which all the multiplicities of weights are $1$ (by I. Kaplansky and J. Santharoubane
\cite{KaSa}) and the classification of unitarizable Harish-Chandra modules (by V. Chari
and A. Pressley \cite{ChP}). Some partial results on Kac's conjecture were obtained in \cite{MP}.

It should be noted that a great number of well-known mathematicians and mathematical physicists contributed to the development of the theory of representations of the Virasoro algebra, for  complete survey we recommend the monograph \cite{Kenji_K}.

In 1992 Benoist answering negatively to Milnor's question \cite{Milnor} on left-invariant affine structures on nilpotent Lie groups presented examples of compact $11$-dimensional  nilmanifolds that carry no complete affine structure. For that he constructed examples of $11$-dimensional nilpotent Lie algebras with no faithfull linear representations of dimension $12$. In his proof \cite{Benoist} he classified
${\mathbb N}$-graded Lie algebras ${\mathfrak a}_r$ defined by two generators $e_1$ and $e_2$ of degrees $1$ and $2$ respectively
and two relations $[e_2,e_3]=e_5$ and $[e_2, e_5]=re_7$, where $r$ is an arbirary scalar and 
$e_{i{+}1}=[e_1, e_i]$ for all $i \ge 2$.

\begin{lemma}[Benoist, \cite{Benoist}]
If $r {\ne}\frac{9}{10},1$, then ${\mathfrak a}_r$ is a finite-dimensional Lie algebra.

1) Let $r=\frac{9}{10}$, then ${\mathfrak a}_r \cong W^+$, positive part of the Witt (Virasoro) algebra, it is infinite-dimensional Lie algebra with the base 
$\tilde e_i, i \ge 1$ and the relations $[\tilde e_i, \tilde e_j]=(j{-}i)\tilde e_{i{+}j}$, where
$\tilde e_i{=}e_i/(i{-}2)!$.

2) Let $r=1$, then ${\mathfrak a}_r \cong {\mathfrak m}_2$.

3) Let $r {\ne} 0,\frac{9}{10},1,2,3$, then ${\mathfrak a}_r$ is a $11$-dimensional filiform Lie algebra.
\end{lemma}

We recall that a finite-dimensional nilpotent Lie algebra ${\mathfrak g}$ is called filiform if it has the 
maximal possible  (for its dimension) value of nil-index $s{(\mathfrak g})=\dim{\mathfrak g}{-}1$. In its turn the nil-index $s{(\mathfrak g})$ is the length of the descending central series of ${\mathfrak g}$.

The idea of Benoist's construction was the following. Benoist considered the algebra  ${\mathfrak a}_{-2}$
of the family ${\mathfrak a}_r$. It is a $11$-dimensional filiform Lie algebra. Itself ${\mathfrak a}_{-2}$ admits a complete affine structure because it is positively graded. Benoist considered its filtered deformation ${\mathfrak a}_{-2,s,t}$ that it is not positively graded and he proved that ${\mathfrak a}_{-2,s,t}$ does not admit any faithfull $12$-dimensional representation. That is an abstruction for existence of a complete affine structure on the corresponding nilmanifold. The proof of non-existence of faithfull $12$-representation was based in particular on the classification of graded faithfull ${\mathfrak a}_{-2}$-modules.

Despite the fact that the positive part $W^+$ of the Witt algebra  was not used in his proof, Benoist suggested that special deformations of finite-dimensional factors of $W^+$ can also be used for counterexamples to Milnor's conjecture in higher dimensions $n > 11$. Hence the classification of 
graded thread $W^+$-modules is quite necessary for the possible proof. 

The aim of this paper is a classification of  $(n+1)$-dimensional graded thread $W^+$-modules with additional structure restriction 
$e_nf_1 \ne 0$
which means that the corresponding  representation of the $n$-dimensional quotient
$W^+/\langle e_{n+1}, e_{n+2},\dots, \rangle$ is faithfull.

Also finite-dimensional graded thread $W^+$-modules played an essential role in explicite constructions of singular Virasoro vectors in  \cite{BdFIZ} and \cite{Mill2}. Besides this graded thread $W^+$-modules were
used for the construction of trivial Massey products in the cohomology $H^*(W^+, {\mathbb K})$ in \cite{FeFuRe},
answering V. Buchstaber's conjecture, that the cohomology $H^*(W^+, {\mathbb K})$ is generated by non-trivial Massey products of one-dimensional cohomology classes.  Finaly V. Buchstaber's conjecture was proved for non-trivial Massey products in \cite{Mill}.

The paper is organized as follows. In the Section 1 we recall necessary definitions and facts on positive part of the Witt (Virasoro) algebra and its graded modules. In particular we introduce the important class of 
$(n{+}1)$-dimensional graded thread $W^+$-modules by a property $e_nf_1 \ne 0$ mentioned above.
We prove  the key Lemma \ref{key_lemma} and its Corollary stating that we have only three types
of $(n+1)$-dimensional graded thread $W^+$-modules defined by basis $f_1,\dots,f_{n+1}$:

a) no zeroes of $e_1$, i.e. $e_1f_i \ne 0, \; i=1,\dots,n$;

b) one zero of $e_1$, i.e. $\exists! k, \; 1 \le k \le n, \; e_1f_k{=}0$;

c)  two neighboring zeroes of $e_1$, i.e. $\exists! k,  1 {\le} k {\le} n{-}1, e_1f_k{=}e_1f_{k{+}1}{=}0$. 

We classify the modules of the type a) (we call them graded thread $W^+$-modules of the type
$(1,1,\dots,1)$) in Section 3. $W^+$-modules with one zero, subcase b), so called modules of the type $(1,\dots,1,0,1,\dots,1)$ are classified in Section 4. We remark that $W^+$-modules of the family $V_{\lambda, \mu}$ are generic modules for both
types. 

Section 5 is devoted to the most interesting case of graded $W^+$-modules with two neighboring zeroes of $e_1$. Modules of this type were applied in \cite{Mill} for the proof of V. Buchstaber's conjecture on Massey products in Lie algebra cohomology $H^*(W^+, {\mathbb K})$ as we mentioned above. 

\section{Definitions and examples}
The Witt algebra $W$ can be defined by its infinite basis 
$e_i=t^{i{+}1}\frac{d}{dt}, i \in {\mathbb Z}$ and the Lie bracket is given by
$$
[e_i,e_j]=(j-i)e_{i+j}, i,j \in {\mathbb Z}.
$$
The Virasoro algebra  Vir is infinite-dimensional Lie algebra, defined by its basis  $\{z, e_i, i \in
{\mathbb Z}\}$ and commutation relations:
$$
[e_i,z]=0, \; \forall i \in {\mathbb Z}, \quad
[e_i,e_j]=(j-i)e_{i+j}+\frac{j^3-j}{12}\delta_{-i,j}z.
$$
Vir is a one-dimensional central extension of the Witt algebra $W$
(the one-dimensional center is spaned by $z$).

There is a subalgebra $W^+ \subset W$ spanned by $e_1, e_2, e_3, \dots,$ basis vectors with positive 
subscripts that we cal the {\it positive part
of the Witt algebra}. 

\begin{definition}[\cite{FeFuRe} ]
\label{thread}
A $W^+$-module $V$ is called graded thread $W^+$-module if there exist two integers $k, N, k < N$, or $k=-\infty, N=+\infty$ and a decompostion $V=\oplus_{{-}\infty}^{{+}\infty}V_j, j \in {\mathbb Z},$
such that:
$$
e_iV_j \subset V_{i{+}j}, i\in {\mathbb N}, j\in {\mathbb Z}, \;\; \dim V_i= \left\{ \begin{array}{cr} 1, & k < i < N,\\ 0, & {\rm otherwise}\\ \end{array} \right.,
$$
\end{definition}
where $e_1, e_2, \dots, e_k, \dots$ is a graded basis of the positive part  $W^+$ of the Witt (Virasoro) algebra.

\begin{remark}
A graded thread $W^+$-module $V=\oplus_{{-}\infty}^{{+}\infty}V_j$ can be defined by its basis
$$f_i,i\in {\mathbb Z}, \;k<i<N, \langle f_i \rangle =V_i,$$ and by the set of its structure constants
$$\alpha_i, \beta_j, \;i,j \in {\mathbb Z}, \; k<i<N{-}2, k<j<N{-}3,$$ such that
$$
e_1f_i=\alpha_i f_{i{+}1}, \; \quad
e_2f_j=\beta_j f_{j{+}2}.
$$
Certainly the constants $\alpha_i, \beta_j$ can not be arbitrary. They must satisfy certain algebraic
relations that we are going to discuss later.
\end{remark}
To begin with, we present two infinite-dimensional examples.
\begin{itemize}
\item One can take $\alpha_i=\alpha,\; \beta_j=\beta$. It is a $W^+$-module $C_{\alpha,\beta}$ where all other
basic elements $e_i, i \ge 3$ act trivially.
\item We have defined basic vectors $e_i$ of $W^+$ as differential operators $e_i=x^{i+1}\frac{d}{dx}$ on
the real (complex) line.
One can consider  the space $V_{\lambda, \mu}$ of tensor densities of the form $P(x)x^{\mu}(dx)^{-\lambda}$, where $P(x)$ is
some polynomial on $x$ and the parameters $\lambda, \mu$ are arbitrary real (complex) numbers. Operator $\xi=f(x)\frac{d}{dx}$ acts
on $F_{\lambda, \mu}$ by means of the Lie derivative $L_{\xi}$:
$$
L_{\xi}P(x)x^{\mu}(dx)^{-\lambda}=\left(f(x)(P(x)x^{\mu})'-\lambda P(x)x^{\mu}f'(x)\right)(dx)^{-\lambda}.
$$
Taking the infinite basis $f_j=x^{j+\mu}(dx)^{-\lambda}$ of $F_{\lambda, \mu}$
we have the following $W^+$-action \cite{Fu}:
$$
e_kf_j=(j+\mu-\lambda(k+1))f_{k+j}.
$$
In other words $V_{\lambda, \mu}$ is a twist of the natural action of $W$ on ${\mathbb C}[t,t^{-1}]$.  The defining set for a $W^+$-module $V_{\lambda, \mu}$ is
$$\alpha_i=i+\mu-2\lambda, \; \beta_j=j+\mu-3\lambda.$$
\end{itemize}
\begin{remark}
The vector space $V_{\lambda, \mu}$ can be regarded as a $W$-module over the entire Witt algebra $W$, or, that is equivalent,
as a zero-energy Virasoro representation \cite{KacR}. As $W$-module $V_{\lambda, \mu}$ is reducible if $\lambda \in {\mathbb Z}$ and
$\beta{=}0$ or $\lambda \in {\mathbb Z}$ and $\beta{=}1$ otherwise it is irreducible \cite{KacR}. The infinite-dimensional $W$-modules $V_{\lambda, \mu}$ and
$V_{\lambda, \mu{+}m}, m \in {\mathbb Z}$ are isomorphic.
\end{remark}

Having an infinite-dimensional graded thread $W^+$-module $V=\langle f_i, i \in {\mathbb Z}\rangle$
one can construct a $(N-k-1)$-dimensional $W^+$-module
$V(k,N)$ taking a subquotient of $V$:
$$
V(k,N) = \oplus_{i>k}^{\infty}V_i/\oplus_{j> N-1}^{\infty}V_j,\quad V(k,N)=\langle f_{k+1},\dots,f_{N-1}\rangle.
$$

From now we deal with a
$(n+1)$-dimensional graded thread $W^+$-module $V$
defined by its basis $ f_1,f_2,\dots,f_{n{+}1}$ and a finite set of constants $\left\{\alpha_1,\dots,\alpha_n, \beta_1,\dots,\beta_{n-1}\right\}$.

The dual module $V^*$ of a finite-dimensional graded thread $W^+$-module $V=\langle f_1,f_2,\dots,f_{n{+}1} \rangle$
in its turn has the structure of a graded thread $W^+$-module with respect to the basis 
$$f'_1=f^{n{+}1},f'_2=f^{n}, \dots,f'_n=f^2, f'_{n{+}1}=f^1,$$
where $f^1,\dots, f^n, f^{n{+}1}$ is the dual basis in $V^*$ with respect to the basis $f_1,\dots, f_n, f_{n{+}1}$ of $V$,  $f^i(f_j)=\delta^i_j$. 
Let $\{\alpha_1,\dots,\alpha_n, \beta_1,\dots,\beta_{n-1}\}$ be the set of structure constants of the $W^+$-module $V$ with respect to the basis 
$f_1,\dots,f_{n{+}1}$.
Then the set
 $$\{-\alpha_n,\dots,-\alpha_1, -\beta_{n-1},\dots,-\beta_1\}$$ defines the structure of the dual $W^+$-module $V^*$:
$$
\begin{array}{l}
e_1f'_{j}=e_1 f^{n{+}2{-}j}={-}\alpha_{n{+}1{-}j} f^{n{+}1{-}j}={-}\alpha_{n{+}1{-}j} f'_{j{+}1},\; j=1,2,\dots,n;\\
e_2f'_{k}=e_2 f^{n{+}2{-}k}={-}\beta_{n{-}k} f^{n{-}k}={-}\beta_{n{-}k} f'_{k{+}2},\; k=1,2,\dots,n-1.
\end{array}
$$
We recall that the dual $W^+$-action on $V^*$ is defined by $(g\cdot f)(x):=-f(gx)$, where $g \in W^+, x \in V, f \in V^*$.

For obvious reasons there is no sense in discussing the irreducibility of graded thread $W^+$-modules: they are all reducible.  Instead of irreducibility one has to discuss indecomposability. 
Let the defining sets of structure constants of a $W^+$-module $V$ be of the following type
$$
\{\alpha_1,\dots, \alpha_m, 0, \alpha_{m{+}1},\dots, \alpha_n, \beta_1,\dots, \beta_{m{-}1},0, 0, \beta_{m{+}2},\dots, \beta_{n-1}\}, \; 0 \le m < n
$$
Then $V$ is a direct sum of two $W^+$-modules $V_1$ and $V_2$
$$
V=V_1 \oplus V_2, \; V_1=\langle f_1,\dots, f_{m{+}1}\rangle, \;  V_2=\langle f_{m{+}2},\dots, f_{n{+}1}\rangle.
$$
where $V_1$ and $V_2$ have the following defining sets of structure constants 
$$
\{ \alpha_1,\dots, \alpha_m, \beta_1,\dots, \beta_{m{-}1}\},\;  \{\alpha_{m{+}1},\dots, \alpha_n, \beta_{m{+}2},\dots, \beta_{n-1}\}, 
$$
respectively. The converse is also true. One can consider a direct sum  $V_1\oplus V_2$ of two finite-dimensional graded thread $W^+$-modules $V_1$ and $V_2$, may be after possible renumbering of the basis vectors from $V_2$.

\begin{lemma}
\label{key_lemma}
Let  $V=\langle f_1,f_2,\dots,f_{n+1} \rangle$ be a
$(n+1)$-dimensional graded thread $W^+$-module such that
$$
\exists k, p, \; 1\le k < k{+}p \le n{+}1, \; p \ge 2, \quad e_1f_k=e_1f_{k{+}p}=0.
$$
Then $e_nf_1=0$.
\end{lemma}
\begin{proof}
The equalities $e_1f_k=e_1f_{k{+}p}=0$ imply that
$$
(p-1)e_{p+1}f_k=\left(e_1e_pf_k -e_pe_1f_k \right)=0.
$$
On the next step we have
$$
pe_{p+2}f_k{=}\left(e_1e_{p+1}f_k -e_{p+1}e_1f_k \right)=0, \;
pe_{p+2}f_{k-1}=\left(e_1e_{p+1}f_{k-1}-e_{p+1}e_1f_{k-1} \right)=0.
$$
One can suppose by an inductive assumption that
$e_{p+s}f_k=\dots=e_{p+s}f_{k-s+1}=0$ for some $s, 1 \le s \le k-1$.
Then it follows that
$$
(p+s-1)e_{p+s+1}f_k=0, \;\dots,
(p+s-1)e_{p+s+1}f_{k{-}s}=\left(e_1e_{p+s}f_{k-s}-e_{p+s}e_1f_{k-s} \right)=0.
$$
Hence $e_{k+p}f_1=\dots=e_{k+p}f_k=0$ and
$$
(k+p-1)e_{k+p+1}f_1=(e_1e_{k+p}f_1-e_{k+p}e_1f_1)=\dots=(k+p-1)e_{k+p+1}f_k=0.
$$
Continuing these calculations we will have that
$e_if_1=\dots=e_if_k=0$ for all $i \ge k{+}p$. In particular $e_nf_1=0$.
\qed
\end{proof}
\begin{corollary}
Let $V=\langle f_1,f_2,\dots,f_{n{+}1} \rangle$ be a
$(n{+}1)$-dimensional graded thread $W^+$-module such that $e_nf_1\ne 0$. Then for its defining set of constants $\alpha_i, i=1,\dots,n$,
 we have three possibilities:
\begin{itemize}
\item[a)] no zeroes, $\alpha_i \ne 0, \; i=1,\dots,n$;
\item[b)] the only one zero, $\exists! k, \; 1 \le k \le n, \; \alpha_k{=}0$;
\item[c)] two neighboring zeroes,
$\exists! k,  1 {\le} k {\le} n{-}1, \alpha_k{=}\alpha_{k{+}1}{=}0$.
\end{itemize}
\end{corollary}

Following \cite{Benoist} we will consider a new basis of $W^+$:
$$
\tilde e_1 =e_1, \quad \tilde e_i = 6(i{-}2)!e_i.
$$
Now we have in particular that
$$
[\tilde e_1,\tilde e_i]=\tilde e_{i{+}1}, \; [\tilde e_2, \tilde
e_3]=\tilde e_5, \; [\tilde e_2,\tilde e_5]=\frac{9}{10} \tilde
e_7.
$$
It was proved in \cite{Benoist} that the Lie algebra generated by two elements $\tilde e_1, \tilde e_2$
with the following two relations on them
$$
\label{def_relat}
\left[\tilde e_2,[\tilde e_1, \tilde e_2]\right]=\left[\tilde e_1,[\tilde e_1,[\tilde e_1, \tilde e_2]]\right],\;
\left[\tilde e_2,\left[\tilde e_1,[\tilde e_1,[\tilde e_1, \tilde e_2]]\right]\right]=
\frac{9}{10}\left[\tilde e_1,\left[\tilde e_1,\left[\tilde e_1,[\tilde e_1,[\tilde e_1, \tilde e_2]]\right]\right]\right]
$$
is isomorphic to $W^+$.

Hence the defining relations (\ref{def_relat}) will give us the following set of equations
\begin{eqnarray*}
R^5_i:\quad\quad ([\tilde e_2, \tilde e_3]-\tilde e_5)f_i=0, \;\; i=1,\dots,n{-}5,\\
R^7_j:\quad ([\tilde e_2, \tilde e_5]-\frac{9}{10}\tilde
e_7)f_i=0, \; i=1,\dots,n{-}7.
\end{eqnarray*}
It is possible to write out the explicit expressions for $R^5_i$
and $R^7_j$ in terms of $\alpha_i, \beta_j$. However rescaling $f_i \to \gamma_if_i$ we can make the
constants $\alpha_i$ equal to one or to zero.

\section{Graded thread $W^+$-modules of the type $(1,1,\dots,1)$.}

Let us consider the
case when all constants $\alpha_i$ of the defining  set for a graded thread $W^+$-module are non-trivial and hence we may
assume (after a suitable rescaling of the basis vectors $f_1,\dots, f_{n+1}$) that $$\tilde e_1f_i=f_{i{+}1}, i=1,\dots, n, \quad \tilde e_2f_j=b_jf_{j{+}2}, j=1,\dots, n-1.$$
Then the equations $R^5_i, R^7_j$ are read as
\begin{equation}
\label{R-uravnenia}
\begin{split}
R^5_i: \quad b_{i{+}3}(b_i{-}b_{i{+}1})-b_i(b_{i{+}2}-b_{i{+}3})=b_i-3b_{i{+}1}+3b_{i{+}2}-b_{i{+}3},\\
R^7_j:\quad
b_{j{+}5}(b_j{-}3b_{j{+}1}{+}3b_{j{+}2}{-}b_{j{+}3})-b_j(b_{j{+}2}{-}3b_{j{+}3}{+}3b_{j{+}4}{-}b_{j{+}5})=\\
=\frac{9}{10}(b_j{-}5b_{j{+}1}{+}10b_{j{+}2}{-}10b_{j{+}3}{+}5b_{j{+}4}{-}b_{j{+}5}),\\
i=1,\dots,n, \; j=1,\dots, n-1.
\end{split}
\end{equation}
There is the relation between $b_i$ and the
initial structure constants $\alpha_i, \beta_j$ of our graded $W^+$-module $V$:
$$
b_i=\frac{6\beta_i}{\alpha_i \alpha_{i+1}}, \quad i=1,\dots,i{-}2.
$$
Consider a module $F_{\lambda, \mu}$. Recall that it has the definig set with $\alpha_i=i+\mu-2\lambda$ and $\beta_j=j+\mu-3\lambda$. We introduce new parameters 
$$
u=\mu-3\lambda, \; v=\mu-2\lambda.
$$
Suppose that $v=\mu-2\lambda \ne {-}1, {-}2,\dots, {-}n.$ Then $F_{\lambda, \mu}$ is of the type $(1,1,\dots,1)$ and 
its coordinates $(b_1,b_2,\dots, b_{n-1})$ are
$$
b_i=6\frac{(u+i)}{(v+i)(v+i+1)}, \; i=1,2,\dots, n-1.
$$
\begin{definition}
An affine (projective) variety  defined by the system  of algebraic equations (\ref{R-uravnenia}) in ${\mathbb K}^{n-1} ({\mathbb P\mathbb K}^{n-2})$ 
is called the affine (projective) variety of $(n+1)$-dimensional graded thread $W^+$-modules of the type $(1,1,\dots,1)$.
\end{definition}
\begin{theorem}
\label{first_main}
Let $V$ be a $(n{+}1)$-dimensional graded thread $W^+$-module of the type $(1,1,\dots,1)$ and 
$n \ge 9$, i.e. $W^+$-module defined by its basis  and defining set of relations
$$
\begin{array}{c}
V=\langle f_1, f_2, \dots, f_{n{+}1} \rangle,\\
e_1f_i=f_{i{+}1}, \; i=1,2,\dots, n;\\
e_2 f_j=b_jf_{j{+}2}, \; j=1,2, \dots, n-1,
\end{array}
$$
Then $V$ is isomorphic to the one and only one  $W^+$-module from the list below.
\begin{itemize}
\item
$V_{\lambda, \mu}(n{+}1), \quad  \mu-2\lambda \ne {-}1, {-}2, \dots, {-}n.$
\item 
$C_{1,x}(n+1),\quad b_1=b_2=\dots=b_{n-1}=x;$
\item
$V_{-2,-3}^t(n+1), \; t \ne 4, \quad b_1=t, b_i=\frac{6(i+3)}{(i+1)(i+2)}, i=2,\dots, n-1;$
\item
$V_{1,3{-}n}^{t}(n+1), \; t \ne 4, \quad b_i=-\frac{6(n-i)}{(n-i-2)(n-i-1)}, i=1,\dots, n-2, b_{n-1}=-t;$
\item
$V_{0,-1}^t(n+1), \; t \ne 6, \quad b_1=t, b_i=\frac{6}{i}, i=2,\dots, n-1;$
\item
$V_{{-}1,{-}2{-}n}^{t}(n+1), \; t \ne 6, \quad b_i=-\frac{6}{(n-i)}, i=1,\dots, n-2, b_{n-1}=t;$
\end{itemize}
\end{theorem}
\begin{remark}
1) A one-parametric family $V_{-2,-3}^t(n+1)$ of graded thread $W^+$-modules is a linear deformation of $V_{-2,-3}(n+1)$. Moreover
$V_{-2,-3}^4(n+1)=V_{-2,-3}(n+1)$. A module $V_{1,3{-}n}^{t}(n+1)$ is dual to $V_{-2,-3}^{t}(n+1)$.

2) Family $V_{0,-1}^t(n+1)$ is one-parametric linear deformation of $V_{0,-1}(n+1)$ and 
$V_{0,-1}^6(n+1)=V_{0,-1}(n+1)$. $V_{1,{-}2{-}n}^{t}(n+1)$ is the dual module to $V_{0,-1}^{t}(n+1)$.

3) $C_{1,x}^*(n+1)=C_{1,-x}(n+1)$;
\end{remark}
\begin{proof}
We prove the Theorem  by induction on dimension $\dim{V}$. 
The equation $R_1^5$ appears first time for a
$5$-dimensional graded thread module. In
dimension $6$ we have  two relations $R_1^5, R_2^5$ and coordinates $(b_1,b_2, \dots, b_5)$ of an arbitrary  $7$-dimensional graded thread $W^+$-module $V$ of the type 
$(1,1,\dots,1)$ satisfies three relations $R_1^5, R_2^5, R_1^7$ respectively.

We start with the classification of $8$-dimensional graded
thread $W^+$-modules. 

\begin{lemma}
\label{main_lemma}
Consider an affine variety $M$ of  $8$-dimensional graded
thread $W^+$-modules of the type $(1,1,\dots,1)$ defined by the following system
of quadratic equations in ${\mathbb K}^6$:
\begin{equation}
\label{main_graded}
 \begin{array}{l}
b_4(b_1-b_2)-b_1(b_3-b_4)=b_1-3b_2+3b_3-b_4,\\
b_5(b_2-b_3)-b_2(b_4-b_5)=b_2-3b_3+3b_4-b_5,\\
b_6(b_3-b_4)-b_3(b_5-b_6)=b_3-3b_4+3b_5-b_6,\\
b_6(b_1{-}3b_2{+}3b_3{-}b_4)-b_1(b_3{-}3b_4{+}3b_5{-}b_6)=\frac{9}{10}(b_1{-}5b_2{+}10b_3{-}10b_4{+}5b_5{-}b_6),
\end{array} 
\end{equation}
then the variety $M$  can be decomposed as the union of the
following two- and one-parametric algebraic subsets:

\begin{itemize}
\label{main_graded_answer}
\item[$M_1:$]  $b_i=6\frac{(u+i)}{(v+i)(v+i+1)},
\; i=1,2,\dots,6,$\; $u\ne v, u\ne v{+}1, \; v \ne
{-}1,{-}2,\dots,{-}7$;
\item[$M_1^0:$] $b_i=\frac{6}{v{+}i},
\; i=1,2,\dots,6,$\; $v\ne {-}1,{-}2,\dots,{-}6$;
\item[$M_2:$] $\begin{array}{l} b_1=\frac{5xy-17x+10y+2}{5y-9}, \; b_2=x, \; b_3=y, \; b_4=y-\frac{2}{5}, \; b_5=\frac{1}{5}\frac{5xy+3x-6}{2x-y+1},\\
b_6=\frac{5xy^2-2xy-22y+10y^2+21x-12}{(2x-y+1)(5y+7)},\; y\ne
\frac{9}{5},{-}\frac{7}{5}, 2x-y+1\ne 0\end{array}$;
\item[$M_3:$]
$b_1=b_2=b_3=b_4=b_5=b_6=t$;
\item[$M_4^{\pm}:$] $\begin{array}{l} b_1=12\pm3\sqrt{19}, b_2=-\frac{2}{5}\pm\frac{1}{5}\sqrt{19}, b_3=\frac{1}{5}\pm\frac{2}{5}\sqrt{19},
b_4=-\frac{1}{5}\pm\frac{2}{5}\sqrt{19}, \\ b_5=\frac{2}{5}\pm\frac{1}{5}\sqrt{19}+t,
b_6=-12\pm3\sqrt{19}+t(\pm\frac{4}{3}\sqrt{19}-\frac{13}{3}).\end{array}$;
\item[$M_5^-:$] $b_1=-\frac{27}{28}, b_2=-\frac{8}{7},
b_3=-\frac{7}{5}, b_4=-\frac{9}{5}, b_5=-\frac{5}{2}, b_6=t$;
\item[$M_5^+:$] $b_1=t, b_2=\frac{5}{2},
b_3=\frac{9}{5}, b_4=\frac{7}{5}, b_5=\frac{8}{7},
b_6=\frac{27}{28}$;
\item[$M_6^+:$] $b_1=t, b_2=\frac{6}{2},
b_3=\frac{6}{3}, b_4=\frac{6}{4},b_5=\frac{6}{5},
b_6=\frac{6}{6}$;
\item[$M_6^-:$] $b_1=-\frac{6}{6}, b_2=-\frac{6}{5},
b_3=-\frac{6}{4}, b_4=-\frac{6}{3}, b_5=-\frac{6}{2}, b_6=-t$;
\end{itemize}
\end{lemma}

\begin{proof}
Denote $b_2=x, b_3=y, b_4=z$ and rewrite after that first two equations of
the system (\ref{main_graded}) 
\begin{equation}
\left\{ \begin{array}{l} b_1(2z-y-1)=xz-3x+3y-z,\\
b_5(2x-y+1)=xz+x-3y+3z.
\end{array} \right.
\end{equation}
Then we multiply the third equation of (\ref{main_graded}) by $(2x-y+1)$ and exclude
$b_5$
$$
(2x-y+1)(2y-z+1)b_6=xyz-4y^2+3yx+6yz-8y-3xz+6z+3x.
$$
Finally we multiply the last equation of (\ref{main_graded}) by
$(2z{-}y{-}1)(2x{-}y{+}1)(2y{-}z{+}1)$ and exclude $b_1, b_5, b_6$. We
will get the following fifth-order equation of
unknowns $x,y,z$:
$$
\left(z-y-\frac{2}{5}\right)F(x,y,z)=0,
$$
where
\begin{equation}
\label{F_function}
F(x,y,z)=y^2(z{-}6)(x{+}6)+y(x{+}z)(xz{+}3z{-}3x{+}36)+3xz(4x{-}4z{-}xz)-9(x{+}z)^2.
\end{equation}
First of all we are going to study an algebraic variety $M_F
\subset {\mathbb K}^3$ defined by the equation $F(x,y,z)=0$. 
Consider  mapping $f: {\mathbb K}^2\backslash \{ (u,v), v=0,{-}1,{-}2,{-}3\} \to {\mathbb K}^3$:
$$
f: (u,v) \to
\left(\frac{6u}{v(v{+}1)},\frac{6(u{+}1)}{(v{+}1)(v{+}2)},\frac{6(u{+}2)}{(v{+}2)(v{+}3)}\right).
$$
\begin{proposition}
The variety $M_F \subset {\mathbb K}^3$ is the
union of  $Im f$ and  three lines $l_1, l_2, l_3$ defined by
$$
l_1: x=y=z;\quad l_2: \left\{ \begin{array}{c} y=3,\\ z=2
\end{array}\right. ,\quad l_3: \left\{ \begin{array}{c} x=-2, \\y=-3,
\end{array}\right.
$$
\end{proposition}
\begin{proof}
Let a point $(x,y,z) \in Im f$. It means that for some  $(u,v), v\ne 0,{-}1,{-}2,{-}3$ we have
\begin{equation}
\label{image_f} \left\{ \begin{array}{l}
6u=x(v^2+v),\\
6u+6=y(v^2+3v+2),\\
6u+12=z(v^2+5v+6). \end{array}\right.
\end{equation}

Consider (\ref{image_f}) as a system of equations with respect to unknowns $u,v$. 
For $z \ne x$  it is equivalent to 
\begin{equation}
\label{image_ff} 
\left\{ \begin{array}{l}
6u=x(v^2+v),\\
(x{-}z)v^2+(x{-}5z)v+12-6z=0,\\
\left((3y{-}x)(z{-}x)-(5z{-}x)(y{-}x)\right)v=(6z{-}12)(y{-}x)-(2y{-}6)(z{-}x).
\end{array}\right.
\end{equation}

1) Let $(3y{-}x)(z{-}x)-(5z{-}x)(y{-}x)=2(2xz-y(x+z))\ne 0$, then substituting
$$
v=\frac{2yz-3zx-6y+3x+yx+3z}{2xz-y(x{+}z)}
$$
in the second equation of (\ref{image_ff}) we get
$$
\frac{(x-z)F(x,y,z)}{(y(x{+}z)-2xz)^2}=0,
$$
where the polynomial $F(x,y,z)$ is defined by (\ref{F_function}).
Hence a point $(x,y,z)$ of the surface $M_F$ with $x\ne z$ and
$y(x{+}z)\ne 2xz$ is in the image $Im f$ and the corresponding
parameters $u,v$ are determined  uniquely.

2) Consider a point $(x,y,z)\in M_F$ such that
\begin{equation}
\label{image_fff} 
\left\{ \begin{array}{l}
(6z-12)(y-x)-(2y-6)(y-x)=0,\\
2xz-y(x+z)=0.
 \end{array} \right.
\end{equation}
This system can be rewritten in a following way
$$ \left\{ \begin{array}{l}
y(z-6)=xz-3x-3z,\\
y(x+z)=2xz.
\end{array} \right.
$$
Substitute $y=\frac{2xz}{x+z}$ ($x+z=0$ implies $x=z=0$) in the first
equation. We obtain $(z{-}x)(xz{+}3z{-}3x){=}0$, hence if $z \ne x$ then
$$
x=\frac{3z}{3-z},\; y=\frac{6z}{6-z}.
$$
Changing parameter $z=\frac{6}{t+2}$ we see that the set of solutions of
(\ref{image_fff}) coincides  with the curve
\begin{equation}
\label{gamma_curve}
\gamma(t)= \left(\frac{6}{t},
\frac{6}{t+1},\frac{6}{t+2}\right).
\end{equation}
Parameters $u,v$ for a point $\left(\frac{6}{t},
\frac{6}{t+1},\frac{6}{t+2}\right)$  of $\gamma$ are determined not
uniquely: $u=v=t$ or $u=v+1=t+2$.

Now let us consider the case $x=z$. Then the system (\ref{image_ff})
is equivalent to the following one
\begin{equation}
\label{image_2} \left\{ \begin{array}{l}
6u=x(v^2+v),\\
(y-x)v^2+(3y-x)v+2y-6=0,\\
2xv=6-3x.
\end{array}\right.
\end{equation}
Substituting $v$  by $\frac{6-3x}{2x}$ in the second equation of
(\ref{image_2}), we'll get
$$
\frac{y(36-x^2)-36x-3x^3}{4x^2}=0.
$$
It follows from the last equation that
$y=\frac{36x+3x^3}{36-x^2}, x\ne 6$.

On the another hand, the square equation (with respect to $y$, we assume also that $x \ne \pm 6$)
$$
F(x,y,x)=y^2(x^2{-}36)+2x(x^2{+}36)y-3x^4-36x^2=0
$$
has two roots $y=x$ and $y=\frac{36x+3x^3}{36-x^2}$. One has to
remark also that if $x=z=\pm 6$, then $y=\pm 6$.

Hence an arbitrary point $P=(x,y,z)$ of the surface $M_F$ with
$x=z \ne y$ also belongs to the image $Im f$. We conclude with a
remark that the system (\ref{image_f}) never has the solutions
with $v=-1, -2$, but it has the solutions with $v=0$ and $v=-3$.
More precisely, if $v=0$ then it follows from the first equation of
(\ref{image_f}) that $u=0$ and this is the case for $y=3,
z=2$ and arbitrary $x$. We get the line $l_2$. 

Analogously the case $v=-3$ implies $u=-2$
and $y=-3, x=-2$. The corresponding set of solutions is $l_3$. 

\end{proof}%proof proposition
\begin{figure}[t]
%\sidecaption[t]
% Use the relevant command for your figure-insertion program
% to insert the figure file.
% For example, with the option graphics use
\includegraphics[scale=.65]{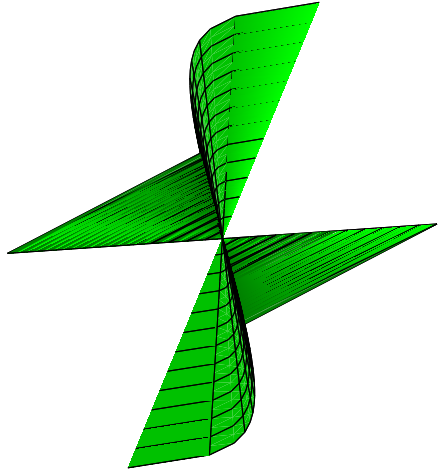}
%
% If no graphics program available, insert a blank space i.e. use
%\picplace{5cm}{2cm} % Give the correct figure height and width in cm
%
%\caption{Please write your figure caption here}
\caption{The surface $M_F$ in ${\mathbb R}^3$.
%If the width of the figure is less than 7.8 cm use the \texttt{sidecapion} command to flush the caption on the left side of the page. If the figure is positioned at the top of the page, align the sidecaption with the top of the figure -- to achieve this you simply need to use the optional argument \texttt{[t]} with the \texttt{sidecaption} command
}
\label{graphic_image}       % Give a unique label
\end{figure}
\begin{remark}
The polynomial $F(x,y,z)$ has the degree two with respect to each
variable $x,y,z$ (thinking other two are parameters). One can verify directly that the curve $\gamma(t)$ defined by (\ref{gamma_curve})
coincides with the set of singular points of $M_F$ (it follows
from the proof of our proposition). Also one can see that $M_F$
has an involution $\sigma_F: M_F \to M_F$:
$$
\sigma_F: (x,y,z) \to (-z,-y,-x).
$$
More precisely 
$$
\sigma_F(f(u,v))=f({-}u{-}2, {-}v{-}4), \; \sigma_F(l_1)=l_1, \;  \sigma_F(l_2)=l_3.
$$
\end{remark}
Now we came to the following natural question. 

{\it Does an
arbitrary point $(x,y,z) \in M_F$ correspond to some solution
$(b_1,x,y,z,b_5,b_6)$ of the initial system (\ref{main_graded})?}

Before answering this question, we state a few preliminary remarks
\begin{proposition}
It exists the unique solution $(b_1,b_2,\dots, b_6)$ of the system
(\ref{main_graded}) with $b_2=b_3=b_4=t$ and it is
$b_1=b_2=b_3=b_4=b_5=b_6=t$.
\end{proposition}
\begin{proof}
If $b_2=b_3=b_4=t$ then the first two equations of (\ref{main_graded}) can be rewritten as
$$
b_1(t-1)=t^2-t, \quad b_5(t+1)=t^2+t.
$$
If $t\ne \pm 1$ then $b_1=b_5=t$ and
$b_6=\frac{t^3+2t^2+t}{(t+1)^2}=t$.

If $t=1$ then $b_5=b_6=1$ and the fourth equation of
(\ref{main_graded}) will look in the following way:
$$
(b_1-1)-\frac{9}{10}(b_1-1)=0.
$$
Hence $b_1=1$. The case $t=-1$ is studied analogously.

\end{proof}
\begin{proposition}
There are no solutions $(b_1,b_2,\dots, b_6)$ of the system
(\ref{main_graded}) such that:
\begin{itemize}
\item[a)] $b_3=3, b_4=2$;
\item[b)] $b_3=-3, b_2=-2$.
\end{itemize}
\end{proposition}
\begin{proof}
Solving first three equations of (\ref{main_graded}) with $b_3=3,
b_4=2$ one get 
$$b_2=7, b_5=\frac{3}{2}, b_6=\frac{6}{5}.$$
But the fourth equation of (\ref{main_graded})  looks as $0\cdot
b_6=-\frac{339}{10}$.

If $b_3=-3$ and $b_2=-2$ we obtain $b_1=-\frac{3}{2}, b_4=-7,
b_6=-9$ and the last equation of  (\ref{main_graded})  will be inconsistent with respect to
$b_5$.

\end{proof}
Now we are going to study the same question for the points of
$Imf$. Namely let suppose that for some choice of $u,v, v \ne
0,-1,-2,-3,$ we have
$$b_2=\frac{6u}{v(v{+}1)}, \; b_3=\frac{6(u{+}1)}{(v{+}1)(v{+}2)},
\; b_4=\frac{6(u{+}2)}{(v{+}2)(v{+}3)}.$$ Then the first three
equations of the system (\ref{main_graded}) we rewrite as
\begin{equation}
\label{new_main}
\begin{split}
b_5\frac{(v{+}4)(6u{+}v^2{-}v)}{v(v{+}1)(v{+}2)}=6\frac{(u{+}3)(6u{+}v^2{-}v)}{v(v{+}1)(v{+}2)(v{+}3)},\\
b_1\frac{(v{-}1)(6u{-}v^2{-}7v)}{(v{+}1)(v{+}2)(v{+}3)}=6\frac{(u{-}1)(6u{-}v^2{-}7v)}{v(v{+}1)(v{+}2)(v{+}3)},\\
b_6\frac{(v{+}5)(6u{+}v^2{+}v{+}6)}{(v{+}1)(v{+}2)(v{+}3)}=
b_5\left(
\frac{6(u{+}1)}{(v{+}1)(v{+}2}+3\right)-6\frac{2uv+5v+3}{(v{+}1)(v{+}2)(v{+}3)}.
\end{split}
\end{equation}
First of all let study a generic case.
\begin{proposition}
Let $v \ne 0,-1,-2,-3$ and moreover
$$
(v{+}4)(6u{+}v^2{-}v)\ne 0, \; (v{-}1)(6u{-}v^2{-}7v)\ne 0,\;
(v{+}5)(6u{+}v^2{+}v{+}6)\ne 0.
$$
Then there exists the only one solution $(b_1,b_5, b_6)$ of
the system (\ref{new_main})
$$b_1=\frac{6(u{-}1)}{(v{-}1)v}, \; b_5=\frac{6(u{+}3)}{(v{+}3)(v{+}4)},
\; b_6=\frac{6(u{+}4)}{(v{+}4)(v{+}5)}.$$
\end{proposition}
\begin{proof}
Direct verification. 
\end{proof}
\begin{proposition}
Let $v \ne 0,-1,-2,-3$. Then if one between of the three given
expressions
$$
6u+v^2-v, \; 6u-v^2-7v, \; 6u+v^2+v+6,
$$
is equal to zero then the rest two of them are non trivial.
\end{proposition}

Now we assume that $v \ne 1,0,-1,-2,-3,-4,-5$.

1) $6u-v^2-7v=0$. Then from the first and the third equations of
(\ref{main_graded}) we have
$$
b_5=\frac{6(u{+}3)}{(v{+}3)(v{+}4)},\;
b_6=\frac{6(u{+}4)}{(v{+}4)(v{+}5)}.
$$
The fourth equation of (\ref{main_graded}) will look as
$$
\frac{1}{10}b_1\frac{v(v{-}1)(v{-}2)(v{-}3)}{(v{+}2)(v{+}3)(v{+}4)(v{+}5)}=
\frac{1}{10}\frac{(v{-}2)(v{-}3)(v^2+7v-6)}{(v{+}2)(v{+}3)(v{+}4)(v{+}5)}.
$$
If $v\ne 2,3$  then
$b_1=\frac{v^2{+}7v{-}6}{v(v{-}1)}=\frac{6(u{-}1)}{v(v{-}1)}$ like
in generic situation.

a) Now let $v=2$ then $u=\frac{v^2+7v}{6}=3$ and we have
$$
b_2=3, b_3=2, b_4=\frac{3}{2}, b_5=\frac{6}{5}, b_6=1.
$$
The component $b_1$ can take an arbitrary value $t$.

b) In the case $v=3$ analogously we obtain $u=5$ and
$$
b_2=\frac{5}{2}, b_3=\frac{9}{5}, b_4=\frac{7}{5},
b_5=\frac{8}{7}, b_6=\frac{27}{28}.
$$
We obtain another one-parametric family of solutions for
(\ref{main_graded}):
$$
(b_1,b_2,\dots,b_6)=\left(t, \frac{5}{2}, \frac{9}{5},
\frac{7}{5}, \frac{8}{7}, \frac{27}{28}\right).
$$

2) $6u+v^2-v=0$. Then $b_1=\frac{6(u{-}1)}{(v{-}1)v}$ and the
third equation of (\ref{main_graded}) will look as
$$
b_6(v{+}5)=b_5(v^2{+}4v{+}3)-v^2+4v+3.
$$
Expressing $b_6$ and rewriting the last equation of
(\ref{main_graded}) we will have
$$
-\frac{1}{10}b_5\frac{(v{+}4)(v{+}3)(v^2{+}8v{-}3)}{(v{+}1)v(v{-}1)}=
\frac{1}{10}\frac{(v^2{+}8v{-}3)(v^2{-}v{-}18)}{v(v{-}1)(v{+}1)}.
$$
If $v^2+8v-3\ne 0$ then $b_5$ is determined uniquely and
hence we have
$$
b_5=\frac{6(u{+}3)}{(v{+}3)(v{+}4)},\;
b_6=\frac{6(u{+}4)}{(v{+}4)(v{+}5)}.
$$
If $v^2+8v-3= 0$, i.e. $v=-4\pm \sqrt{19}$, then
the component $b_5$ can take an arbitrary values and then we have
$u=-\frac{13}{2}\pm\frac{3}{2}\sqrt{19}$ and
\begin{eqnarray*}
b_1=12\pm 3\sqrt{19}, \;\;
b_2=-\frac{2}{5}\pm\frac{1}{5}\sqrt{19},\;\;
b_3=\frac{1}{5}\pm\frac{2}{5}\sqrt{19},\\
b_4=-\frac{1}{5}\pm\frac{2}{5}\sqrt{19}, b_5=t,
b_6=-\frac{46}{3}\pm \frac{10\sqrt{19}}{3}
+t\left(\pm\frac{4}{3}\sqrt{19}-\frac{13}{3}\right).
\end{eqnarray*}
After a parameter change $t \to t+\frac{2}{5}\pm \frac{1}{5}\sqrt{19}$ we
will obtain the final version.

 3) $6u+v^2+v+6=0$. Then $b_1=\frac{6(u{-}1)}{(v{-}1)v}$ and
$b_5=\frac{6(u{+}3)}{(v{+}3)(v{+}4)}$. The fourth equation of the
system (\ref{main_graded}) will look as
$$
-\frac{1}{10}b_6\frac{(v{+}5)(v{+}4)(v{+}7)(v{+}6)}{(v{+}1)(v{+}2)(v{-}1)v}=
\frac{1}{10}\frac{(v{+}7)(v{+}6)(v^2{-}v{-}18)}{v(v{-}1)(v{+}1)(v{+}2)}.
$$
If $v \ne -6, -7$ then $b_6=\frac{6(u{+}4)}{(v{+}4)(v{+}5)}$.

For $v=-6, u=-6$ we get an one-parametric family
$$
(b_1,b_2,\dots,b_6)=\left(-1, -\frac{6}{5}, -\frac{6}{4},
-\frac{6}{3}, -\frac{6}{2}, t\right).
$$
The case $v=-7, u=-8$ corresponds to another one  line of
solutions:
$$
(b_1,b_2,\dots,b_6)=\left(-\frac{27}{28}, -\frac{8}{7},
-\frac{7}{5}, -\frac{9}{5}, -\frac{5}{2}, t\right).
$$
4) Let $v=-4$ then the first equation implies that $u=-3$ or
$u=-\frac{10}{3}$.

a) Let $v=-4, u=-3$, then $b_1=-\frac{6}{5}, b_2=-\frac{3}{2}, b_3=-2,
b_4=-3$. The third equation gives $b_5=-7$ but the fourth equation
becomes inconsistent with respect to $b_6$.

b) Let $v=-4, u=-\frac{10}{3}$, then third and fourth equations of
(\ref{main_graded}) look as
$$
\left\{ \begin{array}{l}
\frac{1}{3}b_6-\frac{2}{3}b_5-\frac{29}{3}=0,\\
\frac{3}{10}b_6-\frac{3}{5}b_5-\frac{2629}{300}=0,\\
\end{array} \right.
$$
and this system is inconsistent with respect to $b_5, b_6$.

5) $v=1$, the the second equation of (\ref{main_graded}) implies
that $u=1$ or $u=\frac{4}{3}$.

a) The case $v=u=1$ corresponds to the family
$$
(b_1,b_2,\dots,b_6)=\left(t, 3, 2, \frac{3}{2}, \frac{6}{5},
1\right).
$$

b) If $v=1, u=\frac{4}{3}$ the fourth equation of
(\ref{main_graded}) degenerates to $-\frac{1}{900}=0$ and hence
the system (\ref{main_graded}) is inconsistent.

6) Let $v=5$, the first equation of (\ref{main_graded}) gives us
$b_5=3(u{+}3)$, the third equation after substitution $v=-5,
b_5=3(u{+}3)$ will look like
$$
\frac{3}{2}u^2+\frac{25}{2}u+26=0.
$$
It has the roots $u=-\frac{13}{3}, -4$.

a) The case $v=-5, u=-4$ corresponds to the family that we already
obtained:
$$
(b_1,b_2,\dots,b_6)=\left(-1, -\frac{6}{5}, -\frac{6}{4},
-\frac{6}{3}, -\frac{6}{2}, t\right).
$$

b) If $v=-5, u=-\frac{13}{3}$ then $b_1=-\frac{16}{15},
b_2=-\frac{13}{10}, b_3=-\frac{5}{3}, b_4=-\frac{7}{3}, b_5=-4$
and the fourth equation is inconsistent with respect to $b_6$.

We studied the solutions of the main system (\ref{main_graded})
that correspond to the points of algebraic variety
$M_F=\{F(x,y,z)=0\}$. Now we will consider the case
$$
b_2=x, \quad b_3=y, \quad b_4=y-\frac{2}{5}=z.
$$
Then one can rewrite the first three equations of
(\ref{main_graded}) as
\begin{equation}
\label{z=y-2/5}
 \left\{ \begin{array}{l}
(y{-}\frac{9}{5})b_1=xy+2y-\frac{17}{5}x+\frac{2}{5},\\
(2x{-}y{+}1)b_5=xy+\frac{3}{5}x-\frac{6}{5},\\
(y{+}\frac{7}{5})b_6=b_5(y{+}3)-2y+\frac{6}{5}.\\
\end{array} \right.
\end{equation}
\begin{proposition}
There are no solutions of  (\ref{z=y-2/5}) with $y=\frac{9}{5},
-\frac{7}{5}$.
\end{proposition}
\begin{proof}
Direct calculations. 
\end{proof}
\begin{proposition}
Let $2x-y+1=0$ then the system (\ref{z=y-2/5}) is consistent if
and only if
$$
x=-\frac{2}{5}\pm\frac{1}{5}\sqrt{19},\;
y=\frac{1}{5}\pm\frac{2}{5}\sqrt{19},\;
z=y-\frac{2}{5}=-\frac{1}{5}\pm\frac{2}{5}\sqrt{19}, b_5=t.
$$
\end{proposition}
\begin{proof}
Indeed $2x-y+1=0$ implies $xy+\frac{3}{5}x-\frac{6}{5}=0$ that is
equivalent to $2x^2+\frac{8}{5}x-\frac{6}{5}=0$. The roots of this
square equation will give us the values given above. Now one can
express $b_1=12\pm 3\sqrt{19}$, take $b_5=t$, then express $b_6$
in terms of $b_5$ taking into account  third equation. Finally we
obtain 
$$
b_6=-\frac{46}{3}\pm \frac{10\sqrt{19}}{3}
+t\left(\pm\frac{4}{3}\sqrt{19}-\frac{13}{3}\right).
$$
That
corresponds to the family of solutions that we have already
obtained.

\end{proof}
Now it is easy to see that in generic situation ($2x-y+1\ne
0$) we have
\begin{eqnarray*}
b_1=\frac{5xy-17x+10y+2}{5y-9},\; b_2=x, \;b_3=y, \; b_4=y-\frac{2}{5},\\
b_5=\frac{1}{5}\frac{5xy+3x-6}{2x-y+1}, \;\;
b_6=\frac{5xy^2-2xy-22y+10y^2+21x-12}{(2x-y+1)(5y+7)}.
\end{eqnarray*}
The last statement concludes the proof of the theorem.
\end{proof} % theorem
\begin{remark}
\label{peresechenie}
We shifted the arguments $u,v$ in our answer
for $M_1$, for instance we considered $b_2=\frac{6u}{v(v{+}1)}$ in
our proof and now $b_2=\frac{6(u{+}2)}{(v{+}2)(v{+}3)}$. One has
to point out some properties of the subsets $M_i$:

\begin{itemize}
\item $M_3$ does not intersect other subsets $M_i$;
\item $P(u,v) {\in} M_1$ belongs to $M_2$ if and only if
$u{=}\frac{1}{30}(v{+}3)(v{+}4)(v{+}5){+}\frac{1}{2}(v{-}3)$;
\item $M_1^{0}$ intersects $M_2$ at $v=\frac{{-}7{\pm}\sqrt{61}}{2}$;
\item $M_4^{\pm}$ intersects $M_1$ at $t{=}{-}\frac{16}{15}\pm \frac{68}{285}\sqrt{19}$
($P_4^{\pm}{=}P({-}\frac{9}{2}{\pm}\frac{3}{2}\sqrt{19},-2{\pm}\sqrt{19})$
and does not intersect other subsets $M_i$;
\item $M_5^{\pm}$ intersects $M_1$ at $t=\pm 4$ (the points $P_5^+{=}P(3,1)$ and $P_5^-{=}P({-}10,{-}9)$ respectively);
\item $M_5^{\pm}$ intersects $M_2$ at $t=0$;
\item $M_6^{\pm}$ intersects $M_1^0$ at $t{=}{\pm} 6$ and $v=0$
and does not intersect other subsets;
\end{itemize}
One can also remark that there exists an involution $\sigma: M \to
M$
$$
\sigma:(b_1,b_2,\dots,b_5,b_6) \to (-b_6,-b_5,\dots,-b_2,-b_1)
$$
with the properties:
\begin{itemize}
\item $M_1$ is invariant with respect to $\sigma$:
$\sigma(P(u,v))=P({-}u{-}7,{-}v{-}8)$;
\item $\sigma(M_1^0)=M_1^0$;
\item $\sigma(M_2 {\cup} M_4^{\pm})=M_2 {\cup} M_4^{\pm}$;
\item $\sigma(M_3)=M_3$;
\item $\sigma(M_5^-)=M_5^+$;
\item $\sigma(M_6^-)=M_6^+$;
\end{itemize}
\end{remark}

\begin{corollary}
\label{main_graded_9} The affine variety of $9$-dimensional graded thread
$W^+$-modules 
$$
\begin{array}{c}
V=\langle f_1, f_2, \dots, f_9 \rangle,\\
e_1f_i=f_{i{+}1}, \; i=1,2,\dots, 8;\\
e_2 f_j=b_jf_{j{+}2}, \; j=1,2, \dots, 7,
\end{array}
$$
can be parametrized by means of two- and one-parametric algebraic subsets
\begin{itemize}
\item[$\tilde M_1:$] $b_i=6\frac{(u+i)}{(v+i)(v+i+1)},
\; i=1,2,\dots,6,7,$\; $u\ne v, u\ne v{+}1, \; v \ne
{-}1,{-}2,\dots,{-}7,{-}8$;
\item[$\tilde M_1^0:$] $b_i=\frac{6}{v{+}i},
\; i=1,2,\dots,6,7,$\; $v\ne {-}1,{-}2,\dots,{-}6,{-}7$;
\item[$\tilde M_2:$] $\begin{array}{c} b_1{=}\frac{(y{-}\frac{3}{5})(y{+}\frac{3}{5})(y{-}2)}{(y-\frac{9}{5})(y-\frac{7}{5})}, 
b_2{=}\frac{(y{-}\frac{8}{5})(y{+}\frac{3}{5})}{(y-\frac{9}{5})},
 b_3{=}y{+}\frac{2}{5},  b_4{=}y,  b_5{=}y{-}\frac{2}{5},\\
b_6{=}\frac{(y{+}\frac{8}{5})(y{-}\frac{3}{5})}{(y{+}\frac{9}{5})},
b_7{=}\frac{(y{-}\frac{3}{5})(y{+}\frac{3}{5})(y{+}2)}{(y{+}\frac{9}{5})(y{+}\frac{7}{5})} \end{array}$;
\item[$\tilde M_3:$]
$\; b_1=b_2=b_3=b_4=b_5=b_6=b_7=t$;
\item[$\tilde M_5^-:$]  $b_1=-\frac{5}{6}, b_2=-\frac{27}{28}, b_3=-\frac{8}{7},
b_4=-\frac{7}{5}, b_5=-\frac{9}{5},b_6=-\frac{5}{2}, b_7=t$;
\item[$\tilde M_5^+:$] $b_1=t, b_2=\frac{5}{2},
b_3=\frac{9}{5}, b_4=\frac{7}{5},b_5=\frac{8}{7},
b_6=\frac{27}{28}, b_7=\frac{5}{6}$;
\item[$\tilde M_6^+:$] $b_1=t, b_2=\frac{6}{2},
b_3=\frac{6}{3}, b_4=\frac{6}{4},b_5=\frac{6}{5}, b_6=\frac{6}{6},
b_7=\frac{6}{7}$;
\item[$\tilde M_6^-:$] $b_1=-\frac{6}{7}, b_2=-\frac{6}{6}, b_3=-\frac{6}{5},
b_4=-\frac{6}{4}, b_5=-\frac{6}{3},b_6=-\frac{6}{2}, b_7=-t$.
\end{itemize}
\end{corollary}
\begin{proof}
The families $M_4^{\pm}$ do not survive to the dimension $9$,
$\tilde M_2$ became an one-parametric family instead of two-parametric $M_2$. $\tilde M_2$ intersects 
$\tilde M_1$ at
$y={\pm}\frac{2}{5}\sqrt{21}$.

\end{proof}
We already classified $9$-dimensional graded thread $W^+$-modules. Let $V$ be 
a $10$-dimensional  $W^+$-module with the basis $f_1,\dots,f_{10}$.
Consider its quotient module $\hat V=V/\langle f_{10}\rangle$ and its submodule $\tilde V=\langle f_2, \dots, f_{10} \rangle$. Both
of them are $9$-dimensional graded thread modules of the type $(1,1,\dots,1)$ with the defining sets $b_1,\dots, b_7$ and $b_2,\dots,b_8$ respectively.

Only two points
$y={\pm}\frac{2}{5}\sqrt{21}$ from $\tilde M_2$ survive to  dimension $10$. These points are in the intersection $\tilde M_1 \cap \tilde M_2$.

Let $(b_1,b_2,\dots,b_{n-2},b_{n-1})$ be coordinates (defining set) of a $(n+1)$-dimensional graded thread $W^+$-module $V=\langle f_1, f_2, \dots, f_n, f_{n+1}\rangle$ of the type $(1,1,\dots,1)$ then its subsets
$(b_1,b_2,\dots,b_{n-2})$ and $(b_2,\dots,b_{n-2},b_{n-1})$ are coordinates of $n$-dimensional quotient 
$V/\langle f_{n+1} \rangle$ and submodule $\langle f_2, \dots, f_n, f_{n+1}\rangle$ respectively. Both of them are
$n$-dimensional graded thread $W^+$-modules of the type $(1,1,\dots,1)$ and we can apply the induction hypothesis. 

For instance let $(b_2,b_3,\dots,b_{n-2},b_{n-1})=(t,3,\dots,\frac{6}{n-2},\frac{6}{n-1})$ then
$(b_1,t,3,\dots,\frac{6}{n-2})$ have also to be present in the Table below and it is impossible. On the another hand
if $(b_1,b_2,\dots,b_{n-2})=(t,3,\dots,\frac{6}{n-2})$ then we have to find the set $(b_2,\dots,b_{n-1})=(\frac{6}{2},\dots,\frac{6}{n-2},b_{n-1})$ in the Table. It can occure if and only if $b_{n-1}=\frac{6}{n-1}$.

\begin{table}
\caption{Graded thread $W^+$-modules of the type $(1,1,\dots,1), \dim{V}=n+1 \ge 10$.}
\label{type_net_nuley}
\begin{center}
\begin{tabular}{|c|c|c|c|c|c|c|c|}
\hline
&&&&&&&\\[-10pt]
module & $b_1$ & $b_2$ & $\dots$ & $b_i$ & $\dots$ & $b_{n-2}$ & $b_{n-1}$ \\
&&&&&&&\\[-10pt]
\hline
&&&&&&&\\[-10pt]
$\begin{array}{c}V_{\lambda, \mu}(n{+}1),\\
u=\mu{-}3\lambda,\\
v=\mu{-}2\lambda,\\v \ne {-}1,\dots, {-}n, \end{array}$ & $\frac{6(u{+}1)}{(v{+}1)(v{+}2)}$ &
$\frac{6(u{+}2)}{(v{+}2)(v{+}3)}$& $\dots$&
$\frac{6(u{+}i)}{(v{+}i)(v{+}i{+}1)}$ & $\dots$ &
$\frac{6(u{+}n{-}2)}{(v{+}n{-}2)(v{+}n{-}1)}$ &$\frac{6(u{+}n{-}1)}{(v{+}n{-}1)(v{+}n)}$ \\
&&&&&&&\\[-10pt]
\hline
&&&&&&&\\[-10pt]
$C_{1,x}(n{+}1)$ & $x$ & $x$ & $\dots$& $x$ & $\dots$ &
$x$ & $x$ \\
&&&&&&&\\[-10pt]
\hline
&&&&&&&\\[-10pt]
$\begin{array}{c}V_{{-}2,{-}3}^t(n{+}1),\\
t\ne 4 \end{array}$ & $t$ & $\frac{5}{2}$& $\dots$&
$\frac{6(i{+}3)}{(i{+}1)(i{+}2)}$ & $\dots$ &
$\frac{6(n{+}1)}{(n{-}1)n}$ &$\frac{6(n{+}2)}{n(n{+}1)}$ \\
&&&&&&&\\[-10pt]
\hline
&&&&&&&\\[-10pt]
$\begin{array}{c}V_{1,3{-}n}^{t}(n{+}1),\\
t\ne 4 \end{array}$ & $-\frac{6(n{+}2)}{n(n{+}1)}$ &
${-}\frac{6(n{+}1)}{(n{-}1)n}$& $\dots$&
${-}\frac{6(n{-}i)}{(n{-}i{-}2)(n{-}i{-}1)}$ & $\dots$ &
${-}\frac{5}{2}$ &${-}t$ \\
&&&&&&&\\[-10pt]
\hline
&&&&&&&\\[-10pt]
$\begin{array}{c}V_{0,{-}1}^t(n{+}1),\\ t \ne 6\end{array}$ & $t$ & $3$& $\dots$& $\frac{6}{i}$ & $\dots$ &
$\frac{6}{(n{-}2)}$ &$\frac{6}{(n{-}1)}$ \\
&&&&&&&\\[-10pt]
\hline
&&&&&&&\\[-10pt]
$\begin{array}{c}V_{{-}1,{-}2{-}n}^{t}(n{+}1),\\ t \ne 6\end{array}$ & ${-}\frac{6}{(n{-}1)}$ & ${-}\frac{6}{(n{-}2)}$& $\dots$&
${-}\frac{6}{(n{-}i)}$ & $\dots$ &
${-}3$ &${-}t$ \\
&&&&&&&\\
\hline
\end{tabular}
\end{center}
\end{table}
\end{proof}

\section{Graded thread modules of the type $(1,\dots,1,0,1,\dots,1)$.}
Now we consider the case, when one and only one of the defining constants $\alpha_i$
vanishes.
$$\exists ! \;k, 1 \le k \le n, \;\alpha_k=0, \; {\rm i.e.}\; e_1f_k=0, \; e_1f_j \ne 0, j\ne k, 1\le j \le n.$$
A graded module of this type is called "d\'ecousu" in \cite{Benoist}. 

\begin{theorem}
\label{theor_odin_nol} Let $V$ be a $(n{+}1)$-dimensional, $n \ge
16$, indecomposable graded thread $W^+$-module of the type $(1,\dots,1,0,1,\dots,1)$,
i.e. there exists a basis $f_i, i=1,\dots,n{+}1,$ of $V$ and $\exists !
k, 1 \le k \le n$, such that
\begin{eqnarray*}
e_if_j \in \langle f_{i+j}\rangle, i+j \le n+1,\; e_if_j =0, i+j > n+1,\\
e_1f_k=0, \; e_1f_i \ne 0, \;i=1,\dots,k{-}1,k{+}1,\dots,n;
\end{eqnarray*}
Then $V$ is isomorphic to one and only one module from the following list

A) $(n+1)$-dimensional quotients of $V_{\lambda, 2\lambda-k}, \lambda \ne 0, {-}1$, 

\begin{itemize}
\item
$V_{\lambda, 2\lambda-k}(n+1), \lambda \ne 0, {-}1,\; 1 \le k \le n;$
\end{itemize}

B) linear deformations of $(n+1)$-dimensional quotients $V_{{-}1,{-}k{-}2}$ and $V_{0,-k}$
\begin{itemize}
\item
$V_{{-}1,{-}k{-}2}^t(n+1),\; 1 \le k < n-1, \; t \in {\mathbb K};$  defined by
\begin{equation}
\begin{split}
e_if_j=\left\{\begin{array}{l} (j{+}i{-}k{-}1)f_{i+j}, j\ne k{+}1, i{+}j \le n{+}1, \\\quad  0, \quad i{+}j > n{+}1; \end{array} \right., \\  e_if_{k{+}1}=
\left\{\begin{array}{l}i(t(i{-}1){-}i{+}2)f_{i+k+1}, i\le n{-}k, \\
\quad 0,\quad i > n{-}k. \end{array} \right.
\end{split}
\end{equation}
\item
$V_{0,-k}^{t}(n+1),  2 <k \le n, \; t \in {\mathbb K};$ it is dual to the $W^+$-module $V_{{-}1,{-}k{-}2}^{-t}(n+1)$
$$V_{0,-k}^{t}(n+1)=V_{{-}1,{-}k{-}2}^{-t*}(n+1);$$ 
\end{itemize}

C) degenerate cases for $k=1, 2, n{-}1, n$.
\begin{itemize}
\item
$V_{0,{-}2}(n+1), \;k=2;$  
\item
$V_{{-}1,{-}4}(n+1)=V_{0,{-}2}^*(n+1), \;k=n{-}1;$
\item
$\tilde V_{0,{-}1}(n+1),\; k=1,$ defined by 
\begin{equation}
\begin{split}
e_if_j=\left\{ \begin{array}{l} (j{-}1)f_{i+j}, j \ge 2, i{+}j \le n{+}1,\\
0, j \ge 2, i+j > n+1,
\end{array} \right.
 e_if_1=\left\{ \begin{array}{l} (i{-}1)f_{i+1}, 1\le i\le n,\\
0, \; i >n. 
\end{array}\right.;
\end{split}
\end{equation}
\item
$\tilde V_{{-}1,{-}2{-}n}(n+1)= \tilde V_{0,{-}1}^*(n+1),\; k=n;$
\item
$\tilde V_{{-}2,{-}3}(n+1), \;k=1,$ defined by 
\begin{equation}
\begin{split}
e_if_j=\left\{ \begin{array}{l} (j{+}2i{-}1)f_{i+j}, j \ge 2, i{+}j \le n{+}1,\\
0, j \ge 2, i+j > n+1,
\end{array} \right.
 e_if_1=\left\{ \begin{array}{l} (i^3-i)f_{i+1}, 1\le i\le n,\\
0, \; i >n. 
\end{array}\right.;
\end{split}
\end{equation}
\item
$\tilde V_{1,{-}n}(n+1)= \tilde V_{{-}2,{-}3}^*(n+1), \;k=n.$
\end{itemize}

\begin{remark}
1) There is another relation of duality 
$$
V_{\lambda, 2\lambda-k}^*(n+1)=V_{{-}\lambda{-}1, {-}2\lambda{-}2-k}(n+1). 
$$

2) $\tilde V_{0,{-}1}(n+1)= \tilde V_{{-}1,{-}2}(n+1)$, where $\tilde V_{{-}1,{-}2}(n+1)$ is defined by 
\begin{equation}
\begin{split}
e_if_j=\left\{ \begin{array}{l} (j{+}i{-}1)f_{i+j}, j \ge 2, i{+}j \le n{+}1,\\
0, j \ge 2, i+j > n+1,
\end{array} \right.
 e_if_1=\left\{ \begin{array}{l} i(i{-}1)f_{i+1}, 1\le i\le n,\\
0, \; i >n. 
\end{array}\right.
\end{split}
\end{equation}
Undeformed modules are non-isomorphic $V_{0,{-}1}(n+1)\neq V_{{-}1,{-}2}(n+1)$. The module  $V_{0,{-}1}(n+1)$ is decomposable
$V_{0,{-}1}(n+1)=\langle f_1 \rangle\oplus \langle f_2,\dots, f_{n{+}1} \rangle$ and $V_{{-}1,{-}2}(n+1)$ is not. However their $n$-dimensional submodules $\langle f_2,\dots, f_{n{+}1} \rangle$ are isomorphic.
\end{remark}

\end{theorem}
\begin{proof}
The first example $V_{\lambda, 2\lambda-k}(n+1)$ in the list of graded thread $W^+$-modules from the Theorem is absolutely ovbious, we have for all $j, 1 \le j \le n$
$$
e_1f_j=(j+2\lambda-k-2\lambda)f_{j+1}=(j-k)f_{j+1}.
$$
What other graded thread $W^+$-modules exist of the type $(1,\dots,1,0,1,\dots,1)$?

\begin{lemma}
Let $V$ be a $(n+1)$-dimensional graded thread $W^+$-module of the type  $(1,\dots,1,0,1,\dots,1)$ with $\alpha_k=0, 1 \le k \le n$, i.e. $e_1f_k=0$.
\label{decomposable} 
The corresponding
graded $W^+$-module $V$  is
decomposable in a direct sum of two graded $W^+$-modules:
$$
V=V_1 \oplus V_2, \;\; V_1=\langle f_1,\dots,f_k \rangle,
V_2=\langle f_{k{+}1},\dots,f_{n{+}1} \rangle.
$$
if and only if $\beta_k=\beta_{k{-}1}=0$, i.e. $e_2f_{k-1}=0, e_2f_k=0$.  
We denoted by $f_1,\dots,f_{n{+}1}$ the graded basis of $V$.
\end{lemma}
\begin{proof}
It is evident that the subspace $V_2$ is invariant. On the another
hand $e_1f_k=e_2f_k=e_2f_{k{-}1}=0$. Hence the subspace $V_1$ is invariant with
respect to $e_1, e_2$ and therefore it is invariant with respect
to the entire $W^+$-action.

\end{proof}
Now we have to rewrite
the basic equations (\ref{R-uravnenia}) for a module with $\alpha_k=0$. It is easy to see that
one have to substitute $b_i$ by $\frac{\beta_i}{\alpha_i
\alpha_{i{+}1}}$ in (\ref{R-uravnenia}) and then multiply $R^5_i$
by the product $\alpha_i\alpha_{i+1}\alpha_{i+2}\alpha_{i+3}$ and
$R^7_i$ by $\alpha_i\alpha_{i+1}\dots\alpha_{i+5}$ respectively.
We suppose that $\alpha_i{=}1, \;i \ne k, \;
\alpha_{k}{=}0$. One can meet $\alpha_{k}$ only in the
denominators of $b_{k}$ and $b_{k{-}1}$. Hence the new
equation that involves $b_{k}$ is obtained from the old one
by a very simple procedure: we keep summands only of the
form $b_{k}b_j$ or $b_{k-1}b_l$, for instance if $1 < k
< n-2$ we have
\begin{equation}
\begin{split}
R^5_{k-1}: \quad b_{k{+}2}(b_{k-1}{-}b_{k}){-}b_{k-1}(b_{k{+}1}{-}b_{k{+}2})=b_{k-1}{-}3b_{k},\\
R^7_{k-1}:\quad
b_{k{+}4}(b_{k-1}{-}3b_{k}){-}b_{k-1}(b_{k{+}1}{-}3b_{k{+}2}{+}3b_{k{+}3}{-}b_{k{+}4})=\frac{9}{10}(b_{k-1}{-}5b_{k}).
\end{split}
\end{equation}

%Now it will be more convinient to consider finite-dimensional graded thread
%$W^+$-modules with  basises indexed not necessary by first natural numbers
%$1,\dots,n$ but generally with 
%$f_m,f_{m+1},\dots,f_{n-1},f_n$ for some integers $m,n$. They are completely determined by their tables of structure constants
%$b_i$:
%$$
%\begin{tabular}{|c|c|c|c|c|}
%\hline
%&&&&\\
%$b_m$ & $b_{m{-}1}$&$\dots$& $b_{n{-}1}$ & $b_{n}$\\
%&&&&\\
%\hline
%\end{tabular}
%$$

It follows from Lemma \ref{decomposable} that for an
indecomposable module $V$ with $\alpha_k=0$ the constants $b_k, b_{k-1}$ can not vanish simultaneously. 
%Hence one may suppose that $b_k \ne 0$, otherwise one can
%consider its dual module $V^*$ with
%$$
%\alpha_{n-k-1}=0, \beta_{n-k-2}=0, \beta_{n-k-1} \ne 0.
%$$

\begin{lemma}
\label{odin_nol} 
Let $V$ be a $9$-dimensional graded thread $W^+$-module
defined by its basis $f_k, f_{k{+}1},\dots, f_{k{+}8},$ and the defining set of relations:
\begin{eqnarray*}
\tilde e_1 f_k=0, \; \tilde e_1 f_i= f_{i{+}1}, \; i=k{+}1,\dots,k{+}7;\\
\tilde e_2 f_j=b_j f_{j+2}, \quad j=k,\dots,k{+}6.
\end{eqnarray*}
Then $V$ is 
either decomposable as a direct sum of $W^+$-modules $\langle f_k \rangle \oplus \langle f_{k{+}1},\dots, f_{k{+}8}\rangle$ ($b_k = 0$) or it is indecomposable  ($b_k \ne 0$) and $V$ isomorphic to the one and only one graded thread $W^+$-module with the defining set $(b_k,\dots,b_{k+6})$ from the
table below
\begin{equation}
\label{odin_noll} 
\begin{tabular}{|c|c|c|c|c|c|c|c|}
\hline
&&&&&&&\\[-10pt]
module & $b_k$ & $b_{k{+}1}$ & $b_{k{+}2}$ & $b_{k{+}3}$ & $b_{k{+}4}$ & $b_{k{+}5}$ & $b_{k{+}6}$  \\
&&&&&&&\\[-10pt]
\hline
&&&&&&&\\[-10pt]
$\begin{array}{c} 1\\ u \neq 1\end{array}$ & $*$ & $\frac{6(u{+}1)}{1\cdot2}$ &
$\frac{6(u{+}2)}{2\cdot3}$ & $\frac{6(u{+}3)}{3\cdot4}$&
$\frac{6(u{+}4)}{4\cdot5}$& $\frac{6(u{+}5)}{5\cdot6}$ &
$\frac{6(u{+}6)}{6\cdot7}$  \\
&&&&&&&\\[-10pt]
\hline
&&&&&&&\\[-10pt]
$2$ & $*$ & $t$ & $\frac{6}{2}$ & $\frac{6}{3}$& $\frac{6}{4}$ &
$\frac{6}{5}$ & $\frac{6}{6}$\\
&&&&&&&\\[-10pt]
\hline
&&&&&&&\\[-10pt]
$3$ & $*$ & $\frac{5}{2}$ & $\frac{9}{5}$ &$\frac{7}{5}$ & $\frac{8}{7}$& $\frac{27}{28}$ & $\frac{5}{6}$ \\
&&&&&&&\\[-10pt]
\hline
&&&&&&&\\[-10pt]
$4$ & $*$ & $1$ & $\frac{9}{5}$ &$\frac{7}{5}$ & $1$& $\frac{3}{4}$ & $\frac{17}{28}$ \\
&&&&&&&\\
\hline
\end{tabular}
\end{equation}
\end{lemma}
\begin{proof}
Consider two equations with $b_k$.
\begin{equation}
\label{9dim}
\begin{split}
R^5_k: \quad b_k(2b_{k{+}3}{-}b_{k{+}2}{-}1)=0,\\
R^7_k:\quad
b_k\left(2b_{k{+}5}{-}b_{k{+}2}{+}3b_{k{+}3}{-}3b_{k{+}4}{-}\frac{9}{10}\right)=0.
\end{split}
\end{equation}
Remark that if $b_k=0$ then for the first basis vector $f_k$ we have
$$
e_1f_k=0, e_2f_k=0,
$$ 
and hence the one-dimensional subspace $\langle f_k \rangle$ is invariant with respect to 
the entire $W^+$-action.

If $b_k \ne 0$ we have two linear equations (\ref{9dim}) on $b_{k{+}2}, b_{k{+}3},
b_{k{+}4}, b_{k{+}5}$. 

Now we can apply the description of $8$-dimensional graded thread modules from the Theorem \ref{main_graded_answer}. We consider the submodule $\langle f_{k{+}1}, \dots, f_{k{+}8}\rangle$ of $V$ as a $8$-dimensional  graded thread module of the type $(1,1,\dots,1)$.

1) Coordinates $b_{k{+}i}=\frac{6(u{+}i)}{(v{+}i)(v{+}i{+}1)}, i=1,\dots,6,$ of a point $P(u,v) \in M_1$
 satisfy both equations $R^5_k$ and $R^7_k$ if and only if
\begin{itemize}
\item $u{=}4, v{=}2$, i.e. $b_{k{+}1}{=}\frac{5}{2}, b_{k{+}2}{=}\frac{9}{5}, b_{k{+}3}{=}\frac{7}{5}, b_{k{+}4}{=}\frac{8}{7},
 b_{k{+}5}{=}\frac{27}{28}, b_{k{+}6}{=}\frac{5}{6}$;
\item $v{=}0$, i.e. $b_{k{+}i}{=}\frac{6(u{+}i)}{i(i{+}1)}, \;
i=1,\dots,6, \; u\ne 0,1$.
\end{itemize}

2) Coordinates $b_{k{+}i}{=}\frac{6}{(v{+}i)}$ of a point $P(v) \in M_1^0$
 satisfy $R^5_k$ and $R^7_k$ also in two cases
\begin{itemize}
\item $v=0$, i.e. $b_{k{+}1}{=}6, b_{k{+}2}{=}3, b_{k{+}3}{=}2, b_{k{+}4}{=}\frac{3}{2}, b_{k{+}5}{=}\frac{6}{5}, b_{k{+}6}{=}1$;
\item $v=1$, i.e. $b_{k{+}1}{=}3, b_{k{+}2}{=}2, b_{k{+}3}{=}\frac{3}{2}, b_{k{+}4}{=}\frac{6}{5}, b_{k{+}5}{=}1, b_{k{+}6}{=}\frac{6}{7}$.
\end{itemize}
Hence we have to remove the restriction $u \ne 0,1$ in the first line of our table.

3) There are only two points in
$M_2$ satisfying $R^5_k$ and $R^7_k$:
\begin{itemize}
\item $x{=}\frac{5}{2}, y{=}\frac{7}{4}$, i.e. $b_{k{+}1}{=}\frac{9}{2}, b_{k{+}2}{=}\frac{5}{2}, b_{k{+}3}{=}\frac{7}{4}, b_{k{+}4}{=}\frac{27}{20},
b_{k{+}5}{=}\frac{11}{20}, b_{k{+}6}{=}\frac{13}{14}$;
\item $x{=}\frac{9}{5}, y{=}\frac{7}{5}$, i.e. $b_{k{+}1}{=}1, b_{k{+}2}{=}\frac{9}{5}, b_{k{+}3}{=}\frac{7}{5}, b_{k{+}4}{=}1,
 b_{k{+}5}{=}\frac{3}{4}, b_{k{+}6}{=}\frac{17}{28}$.
\end{itemize}

But the point with parameters $x{=}\frac{5}{2}, y{=}\frac{7}{4}$ coincides with
the point $P(\frac{1}{2},0) \in M_1$ with $u{=}\frac{1}{2}, v{=}0$
that we have already considered above. 

4) Coordinates $(b_{k{+}1},b_{k{+}2},\dots, b_{k{+}6})$ of a point $P(t) \in M_3$
satisfies $R^5_k$ only if $t=1$. However $P(1)$ does not satisfy
$R^7_k$. 

5) It is easy to verify directly
that there is no point in the subsets $M_4^{\pm}$, $M_5^{\pm}$,$M_6^-$
that satisfies the equation $R^5_k$.  

6) All points from $M_6^+$ satisfy both equations (\ref{9dim}). 
However it does not hold for points from $M_6^-$.

\end{proof}

\begin{corollary}
\label{odin_nol_corollary}
Let $V$ be a idecomposable $(n+1)$-dimensional graded thread $W^+$-module
defined by its basis $f_1, f_2,\dots, f_{n{+}1}, n+1\ge 10$ and the defining set of relations:
\begin{eqnarray*}
\tilde e_1 f_1=0, \; \tilde e_1 f_i= f_{i{+}1}, \; i=1,\dots,n;\\
\tilde e_2 f_j=b_j f_{j+2}, \quad j=1,\dots,n{-}1.
\end{eqnarray*}
then $V$ is isomorphic to the one and only one $W^+$-module from the table below
$$
\begin{tabular}{|c|c|c|c|c|c|c|c|}
\hline
&&&&&&&\\[-10pt]
module & $b_1$ & $b_2$ & $b_3$ & $\dots$ & $b_{i{+}1}$ &  $\dots$ & $b_{n{-}1}$  \\
&&&&&&&\\[-10pt]
\hline
&&&&&&&\\[-10pt]
$\begin{array}{c}V_{\lambda,{-}1{+}2\lambda},\\ \lambda \ne 0,{-}1\end{array}$ & ${-}6\lambda$ & $\frac{6({-}\lambda{+}1)}{1\cdot2}$ &
$\frac{6(-\lambda{+}2)}{2\cdot3}$ & $\dots$& $\frac{6(-\lambda{+}i)}{i(i{+}1)}$&
$\dots$ &
$\frac{6(-\lambda{+}n{-}2)}{(n{-}2)(n{-}1)}$  \\
&&&&&&&\\[-10pt]
\hline
&&&&&&&\\[-10pt]
$\tilde V_{0,{-}1}$ & $6$ & $\frac{6}{2}$ &
$\frac{6}{3}$ & $\dots$& $\frac{6}{i{+}1}$&
$\dots$ &
$\frac{6}{n{-}1}$  \\
&&&&&&&\\[-10pt]
\hline
&&&&&&&\\[-10pt]
$V_{{-}1,{-}3}^t$ & $6$ & $t$ & $3$ & $\dots$& $\frac{6}{i}$ &
$\dots$ & $\frac{6}{n{-}2}$\\
&&&&&&&\\[-10pt]
\hline
&&&&&&&\\[-10pt]
$\tilde V_{{-}2,{-}3}$ & $6$ & $\frac{5}{2}$ & $\frac{9}{5}$ &$\dots$ & $\frac{6(4{+}i)}{(2{+}i)(3{+}i)}$& $\dots$ & $\frac{6(n{+}2)}{n(n{+}1)}$ \\
&&&&&&&\\
\hline
\end{tabular}
$$
\end{corollary}
\begin{proof}
We rescaled the first vector $f_1$ in order to fix the value of $b_1$ (we recall that $b_1$ is not determined by equations $R_1^5$ and $R^7_1$). It is possible because $e_1f_1=0$. 

We can conclude that if $b_k \ne 0$ then 
$b_{k{+}1},\dots,b_{k{+}6}$ have the values prescribed by 
Lemma \ref{9dim}. It follows from Theorem \ref{first_main}
that $b_{k{+}i}, 1 \le i \le s$  if $7 \le s\le n-9$ is
determined uniquely for all three subcases of (\ref{odin_noll})
$$
1)\;  b_{k{+}i}{=}\frac{6(u{+}i)}{i(i{+}1)}; \quad 2)\;
b_{k{+}i}{=}\frac{6}{i}, \; ;\quad 3)\;
b_{k{+}i}{=}\frac{6(i{+}4)}{(i{+}2)(i{+}3)}.
$$

However for the subcase
$$
4) \;\; b_k{=}1, \;b_{k{+}1}{=}1,  b_{k{+}2}{=}\frac{9}{5},
\;b_{k{+}3}{=}\frac{7}{5}, \;b_{k{+}4}{=}1,
\;b_{k{+}5}{=}\frac{3}{4}, \;b_{k{+}6}{=}\frac{17}{28}.
$$
the defining set can not be extended to a system $\{b_{k+1},\dots,b_{k{+}7}\}$.
It follows from the fact that 4) corresponds to the point 
$\left(1,\frac{9}{5},\frac{7}{5},1,\frac{3}{4},\frac{17}{28}\right)$ of the
subset $M_2$ defined in Theorem \ref{first_main} by parameters
$x{=}y{+}\frac{2}{5}{=}\frac{9}{5}, y{=}\frac{7}{5}$. This set  can not
be extended to $\left(1,\frac{9}{5},\frac{7}{5},1,\frac{3}{4},\frac{17}{28}, b_{k+7}\right)$. One can also very it directly considering the equations $R^5_{k{+}7}, R^7_{k{+}7}$.

1) For convenience of notations we denote by $V_{{-}1,{-}3}^t$ a linear deformation of $V_{{-}1,{-}3}$ ($t$ is a parameter). It is defined by the formulas
\begin{eqnarray*}
e_if_j=(j+i-2)f_{i+j}, j\ne 2, i+j \le n+1,\\
 e_if_2=i(t(i-1)-i+2)f_{i+2}, i\le n-1.
\end{eqnarray*}

2) The $W^+$-module $\tilde V_{0,{-}1}$ is also deformed $W^+$-module $V_{0,{-}1}$. It is defined by
\begin{eqnarray*}
e_if_j=(j+i+1)f_{i+j}, j \ge 2, i+j \le n+1, \\ \; e_if_1=(i{-}1)f_{i+1}, 1 \le i \le  n.
\end{eqnarray*}

3) The  $W^+$-module  $\tilde V_{{-}2,{-}3}$ is defined by
\begin{eqnarray*}
e_if_j=(j+2i-1)f_{i+j}, j \ge 2, i+j \le n+1, \\ e_if_1=(i^3-i)f_{i+1}, 1\le i\le n.
\end{eqnarray*}
We also substituted for convinience $u=-\lambda$.

\end{proof}

\begin{lemma}
\label{one_zero}
Let $V$ be a $k+8$-dimensional graded thread $W^+$-module
defined by its basis $f_1,\dots,f_k, f_{k{+}1},\dots, f_{k{+}8}, 2\le k \le 8,$ and the defining set of relations:
\begin{eqnarray*}
\tilde e_1 f_k=0, \; \tilde e_1 f_i= f_{i{+}1}, \; i=1,\dots, k{-}1, k{+}1,\dots,k{+}7;\\
\tilde e_2 f_j=b_j f_{j+2}, \quad j=1,\dots,k{+}6.
\end{eqnarray*}
Then $V$ is 
either decomposable as a direct sum of $W^+$-modules $\langle f_1,\dots, f_k \rangle \oplus \langle f_{k{+}1},\dots, f_{k{+}8}\rangle$ ($b_k = 0$) or it is indecomposable  ($b_k \ne 0$) and $V$ isomorphic to the one and only one graded thread $W^+$-module with the defining set $(b_1,\dots,b_{k+6})$ from the
table below
\begin{equation}
\label{table_k+8}
\begin{tabular}{|c|c|c|c|c|c|c|c|c|c|c|}
\hline
&&&&&&&&&&\\[-10pt]
module &$b_{k-i}$ & $\dots$ & $b_{k-2}$&$b_{k{-}1}$ & $b_k$ & $b_{k{+}1}$ & $\dots$ & $b_{k{+}i}$ & $\dots$ &  $b_{k{+}6}$  \\
&&&&&&&&&&\\[-10pt]
\hline
&&&&&&&&&&\\[-10pt]
$1$& $\frac{6(u-i)}{i(i-1)}$& $\dots$&$\frac{6(u{-}2)}{1\cdot2}$& $-6(u{-}1)$ & $6u$ & $\frac{6(u{+}1)}{1\cdot2}$ &
$\dots$ & $\frac{6(u{+}i)}{i(i{+}1)}$&
$\dots$ &
$\frac{6(u{+}6)}{6\cdot7}$  \\
&&&&&&&&&&\\[-10pt]
\hline
&&&&&&&&&&\\[-10pt]
$2$&${-}\frac{6}{i}$& $\dots$& ${-}\frac{6}{2}$&$0$ & $6$ & $t$ & $\dots$ & $\frac{6}{i}$& $\dots$ &
 $\frac{6}{6}$\\
&&&&&&&&&&\\
\hline
&&&&&&&&&&\\[-10pt]
$3$&${-}\frac{6}{i-1}$&$\dots$ & $t$ &${-}6$ & $0$ & $\frac{6}{2}$ & $\dots$ & $\frac{6}{i+1}$ &
$\dots$ & $\frac{6}{7}$\\
&&&&&&&&&&\\
\hline
\end{tabular}
\end{equation}
\end{lemma}
\begin{proof}
1) Consider 
equations $R^5_{k{-}1}$ and $R^7_{k{-}1}$
\begin{equation}
\label{equations_k-1}
\begin{split}
R^5_{k{-}1}: \quad b_{k{-}1}(2b_{k{+}2}{-}b_{k{+}1}{-}1)=b_k(b_{k{+}2}{-}3),\\
R^7_{k{-}1}:\quad
b_{k{-}1}\left(2b_{k{+}4}{-}b_{k{+}1}{+}3b_{k{+}2}{-}3b_{k{+}3}{-}\frac{9}{10}\right)=3b_k\left(b_{k{+}4}{-}\frac{3}{2}\right).
\end{split}
\end{equation}

In the subcase 1) from the Table (\ref{odin_noll}) both equations are equivalent to
$${-}b_{k{-}1}u=b_k(u{-}1).$$ 
Hence one can take
$b_k=6u\gamma, \; b_{k{-}1}=-6(u{-}1)\gamma, \; u \ne 0,1, \gamma \ne 0$. If $u=0$ then $b_k=0$, if
$u=1$ then $b_{k-1}=0$. After (if necessary) rescaling of $f_{k}$ we may assume that 
$\gamma=1$.  

In the subcase 2) we have $b_k\ne 0, b_{k{+}1}{=}t, b_{k{+}2}{=}\frac{6}{2}, b_{k{+}3}{=}\frac{6}{3},
b_{k{+}4}{=}\frac{6}{4}$. Then $R^5_{k{-}1}$ and $R^7_{k{-}1}$ are
$$
b_{k{-}1}\left(5-t\right)=0, \quad
b_{k{-}1}\left(\frac{51}{10}-t\right)=0.
$$
That implies $b_{k{-}1}=0$. We set $b_k=6$.

For the subcase 3) in (\ref{odin_noll}) the system $R^5_{k{-}1}$ and $R^7_{k{-}1}$ look
$$
b_{k-1}\frac{1}{10}=-b_k\frac{6}{5}, \quad b_{k-1}\frac{3}{35}=-b_k\frac{15}{14}
$$
and it is inconsistent. Hence the case 3) is not extendable to the left.

The subcase 4) also leeds to a inconsistent system on unknowns $b_{k-1}$ and $b_k$ and also is not extendable.

Now we have to study the case $b_k{=}0$. As $b_{k{-}1}\ne 0$ we remark 
that (\ref{equations_k-1}) is equivalent 
to the following linear system
\begin{equation}
\label{10dim}
\begin{split}
2b_{k{+}2}{-}b_{k{+}1}{-}1=0,\\
2b_{k{+}4}{-}b_{k{+}1}{+}3b_{k{+}2}{-}3b_{k{+}3}{-}\frac{9}{10}=0.
\end{split}
\end{equation}
One can remark that the equations (\ref{10dim}) can be obtained from (\ref{9dim}) just by shifting the index $k \to k-1$. The
mimic of the proof of the Lemma \ref{odin_nol} will give us
the following answer ($b_{k{-}1}$ can take arbitrary values but after rescaling (if necessary) we may assume that $b_{k{-}1}{=}{-}6$).
$$
b_k{=}0, b_{k{+}1}{=}\frac{6}{2}, b_{k{+}2}{=}\frac{6}{3}, b_{k{+}3}{=}\frac{6}{4},
b_{k{+}4}{=}\frac{6}{5}, b_{k{+}5}{=}\frac{6}{6}, b_{k{+}6}{=}\frac{6}{7}.
$$

We summarize our results by means of the following table:
\begin{equation}
\label{table10}
\begin{tabular}{|c|c|c|c|c|c|c|c|c|}
\hline
&&&&&&&&\\[-10pt]
module & $b_{k{-}1}$ & $b_k$ & $b_{k{+}1}$ & $b_{k{+}2}$ & $b_{k{+}3}$ & $b_{k{+}4}$ & $b_{k{+}5}$ & $b_{k{+}6}$  \\
&&&&&&&&\\[-10pt]
\hline
&&&&&&&&\\[-10pt]
$1$ & $-6(u{-}1)$ & $6u$ & $\frac{6(u{+}1)}{1\cdot2}$ &
$\frac{6(u{+}2)}{2\cdot3}$ & $\frac{6(u{+}3)}{3\cdot4}$&
$\frac{6(u{+}4)}{4\cdot5}$& $\frac{6(u{+}5)}{5\cdot6}$ &
$\frac{6(u{+}6)}{6\cdot7}$  \\
&&&&&&&&\\[-10pt]
\hline
&&&&&&&&\\[-10pt]
$2$ & $0$ & $6$ & $t$ & $3$ & $2$& $\frac{3}{2}$ &
$\frac{6}{5}$ & $1$\\
&&&&&&&&\\
\hline
&&&&&&&&\\[-10pt]
$3$ & ${-}6$ & $0$ & $3$ & $2$ & $\frac{3}{2}$ &
$\frac{6}{5}$ & $1$ &$\frac{6}{7}$\\
&&&&&&&&\\
\hline
\end{tabular}
\end{equation}

2) Suppose now that
$k \ge 3$. Hence we may consider 
equations $R^5_{k{-}2}$ and $R^7_{k{-}2}$
\begin{equation}
\label{eq_k-2}
\begin{split}
R^5_{k{-}2}: \quad b_{k{-}2}b_k=3b_{k{-}1}{-}3b_k{-}b_{k{-}1}b_{k{+}1},\\
R^7_{k{-}2}:\quad
b_{k{-}2}b_k=\frac{9}{2}b_{k{-}1}{-}9b_k{-}3b_{k{-}1}b_{k{+}3}{+}3b_kb_{k{+}3}.
\end{split}
\end{equation}

For the module 1) in (\ref{table10}) the equations  (\ref{eq_k-2}) are equivalent to one equation 
$$b_{k{-}2}u=3u(u{-}2).$$ 
Hence we  take
$b_{k{-}2}=3(u{-}2)=\frac{6(u{-}2)}{1\cdot2}$ if $u \ne 0$.

For the module 2) in (\ref{table10})  our equations  imply $b_{k{-}2}={-}3$. 

In the subcase 3) taking into account  $b_k{=}0$ we conclude that  all values
of $b_{k{-}2}$ are valid and we set $b_{k{-}2}=t$.

3) Suppose now that $k \ge 4$.  The equations on $b_{k{-}3}$ are the following ones
\begin{equation}
\begin{split}
R^5_{k{-}3}: \quad b_{k{-}3}(2b_k{-}b_{k{-}1})=b_{k{-}2}b_k{+}3b_{k{-}1}{-}b_k,\\
R^7_{k{-}3}:\quad
b_{k{-}3}(b_{k{-}1}-3b_k)=b_{k{+}2}(3b_{k{-}1}{-}b_k){-}9b_{k{-}1}{+}9b_k,
\end{split}
\end{equation}
For the first module 1) they
are equivalent to
$$
b_{k{-}3}(3u{-}1)=(3u{-}1)(u{-}3), \quad b_{k{-}3}({-}4u{+}1)=({-}4u{+}1)(u{-}3).
$$
Hence $b_{k{-}3}=u{-}3$. For the second module 2)  both equations are equivalent to $b_{k{-}3}={-}2$.
For the module 3) we have $b_{k{-}3}={-}3$.

\begin{table}
\caption{Graded thread $W^+$-modules of the type $(1,\dots,1,0,1\dots,1), \dim{V}=n+1 \ge 16$.}
\label{type_odin_nul}
\begin{tabular}{|c|c|c|c|c|c|c|c|c|c|}
\hline
&&&&&&&&&\\[-10pt]
module & $b_{1}$& $b_2$ & $\dots$& $b_{k{-}2}$ & $b_{k{-}1}$ & $b_{k}$ 
  & $b_{k{+}1}$ & $\dots$ &$b_{n{-}1}$  \\
&&&&&&&&&\\[-10pt]
\hline
&&&&&&&&&\\[-10pt]
$\begin{array}{c} V_{\lambda, 2\lambda-k}(n+1),\\ \lambda \ne0,{-}1,\\
2<k<n{-}1\end{array}$ & $\frac{6({-}\lambda{-}k{+}1)}{(k{-}1)(k{-}2)}$ & $\frac{6({-}\lambda{-}k{+}2)}{(k{-}2)(k{-}3)}$& $\dots$
&$\frac{6({-}\lambda{-}2)}{2\cdot 1}$ & $6({-}\lambda{-}1)$& ${-}6\lambda$ &
$\frac{6({-}\lambda{+}1)}{1\cdot 2}$ &
 $\dots$ & $\frac{6({-}\lambda{+}n{-}1{-}k)}{(n{-}k{-}1)(n{-}k)}$ \\
&&&&&&&&&\\[-10pt]
\hline
&&&&&&&&&\\[-10pt]
$\begin{array}{c} V_{\lambda, 2\lambda-2}(n+1),\\ \lambda \ne  {-}1, k=2 \end{array}$ & $6(\lambda{+}1)$ & ${-}6\lambda$& $\dots$
&$\dots$ & $\frac{6({-}\lambda{+}k{-}3)}{(k{-}3)(k{-}2)}$& $\dots$ &
$\frac{6({-}\lambda{+}k{-}1)}{(k{-}1)k}$ &
 $\dots$ & $\frac{6({-}\lambda{+}n{-}3)}{(n{-}3)(n{-}2)}$ \\
&&&&&&&&&\\[-10pt]
\hline
&&&&&&&&&\\[-10pt]
$\begin{array}{c} V_{\lambda, 2\lambda-1}(n+1),\\ \lambda \ne 0,{-}1, k=1 \end{array}$ & $-6\lambda$ &$\frac{6({-}\lambda{+}1)}{1\cdot2}$ & $\dots$
&$\dots$ & $\frac{6({-}\lambda{+}k{-}2)}{(k{-}2)(k{-}1)}$& $\dots$ &
$\frac{6({-}\lambda{+}k)}{k(k{+}1)}$ &
 $\dots$ & $\frac{6({-}\lambda{+}n{-}2)}{(n{-}2)(n{-}1)}$ \\
&&&&&&&&&\\[-10pt]
\hline
&&&&&&&&&\\[-10pt]
$\begin{array}{c} V_{\lambda, 2\lambda-n}(n+1),\\ \lambda \ne 0,{-}1, k=n \end{array}$ &$\frac{6({-}\lambda{-}n{+}1)}{(n{-}2)(n{-}1)}$  &$\frac{6({-}\lambda{-}n{+}2)}{(n{-}2)(n{-}3)}$ & $\dots$
&$\dots$ & $\frac{6({-}\lambda{-}n{+}k{-}1)}{(n{-}k{+}2)(n{-}k{+}1)}$& $\dots$ &
$\frac{6({-}\lambda{-}n{+}k)}{(n{-}k)(n{-}k{-}1)}$ &
 $\dots$ &$6({-}\lambda{-}1)$  \\
&&&&&&&&&\\[-10pt]
\hline
&&&&&&&&&\\[-10pt]
$\begin{array}{c} V_{\lambda, 2\lambda-n+1}(n+1),\\ \lambda \ne 0, k=n-1 \end{array}$ & $\frac{6({-}\lambda{-}n{+}2)}{(n{-}2)(n{-}3)}$ & $\frac{6({-}\lambda{-}n{+}3)}{(n{-}3)(n{-}4)}$& $\dots$
&$\dots$ & $\frac{6({-}\lambda{-}n{+}k)}{(n{-}k)(n{-}k{-}1)}$& $\dots$ &
$\frac{6({-}\lambda{-}n{+}k{+}2)}{(n{-}k{-}2)(n{-}k{-}3)}$ &
 $\dots$ & $-6\lambda$ \\
&&&&&&&&&\\[-10pt]
\hline
&&&&&&&&&\\[-10pt]
$\begin{array}{c}V_{{-}1,{-}k{-}2}^t(n+1), \\ 2< k <n{-}1\end{array}$ & ${-}\frac{6}{k{-}1}$ & ${-}\frac{6}{k-2}$ & $\dots$ & ${-}\frac{6}{2}$& $0$& $6$ & $t$ &
$\dots$&$\frac{6}{n{-}k{-}1}$ \\
&&&&&&&&&\\[-10pt]
\hline
&&&&&&&&&\\[-10pt]
$\begin{array}{c} V_{{-}1, {-}4}^t(n+1),\\ k=2 \end{array}$ & $0$ & $6$& $\dots$
&$\frac{6}{k{-}4}$& $\frac{6}{k{-}3}$& $\frac{6}{k{-}2}$ &
$\frac{6}{k{-}1}$ &
 $\dots$ & $\frac{6}{n{-}3}$ \\
&&&&&&&&&\\[-10pt]
\hline
&&&&&&&&&\\[-10pt]
$\begin{array}{c}V_{{-}1,{-}3}^{t}(n+1),\\  k=1 \\\end{array}$ & $6$& $t$ & $\dots$ & $\frac{6}{k-3}$& $\frac{6}{k-2}$& $\frac{6}{k-1}$ &
$\frac{6}{k}$ &
$\dots$&$\frac{6}{n-2}$ \\
&&&&&&&&&\\[-10pt]
\hline
&&&&&&&&&\\[-10pt]
$\begin{array}{c}V_{0,-k}^{t}(n+1),\\  2 <k <n-1 \\\end{array}$ & ${-}\frac{6}{(k{-}2)}$& ${-}\frac{6}{(k{-}3)}$ & $\dots$ & $t$& ${-}6$& $0$ &
$\frac{6}{2}$ &
$\dots$&$\frac{6}{n{-}k}$ \\
&&&&&&&&&\\[-10pt]
\hline
&&&&&&&&&\\[-10pt]
$\begin{array}{c}V_{0,-n{+}1}^t(n+1),\\  k=n-1 \\\end{array}$ & ${-}\frac{6}{n-3}$& ${-}\frac{6}{n-2}$ & $\dots$ & ${-}\frac{6}{n{-}k}$& ${-}\frac{6}{n{-}k{-}1}$& ${-}\frac{6}{n{-}k{-}2}$ &
${-}\frac{6}{n{-}k{-}3}$ &
$\dots$&$0$ \\
&&&&&&&&&\\[-10pt]
\hline
&&&&&&&&&\\[-10pt]
$\begin{array}{c}V_{0,-n}^t(n+1),\\  k=n \\\end{array}$ & ${-}\frac{6}{n-2}$& ${-}\frac{6}{n-1}$ & $\dots$ & ${-}\frac{6}{n{-}k{+}1}$& ${-}\frac{6}{n{-}k}$& ${-}\frac{6}{n{-}k{-}1}$ &
${-}\frac{6}{n{-}k{-}2}$ &
$\dots$&${-}6$ \\[-10pt]
&&&&&&&&&\\
\hline
&&&&&&&&&\\[-10pt]
$\begin{array}{c}\tilde V_{0,{-}1}(n+1),\\  k=1 \\\end{array}$ & $6$& $\frac{6}{2}$ & $\dots$ & $\frac{6}{k{-}2}$& $\frac{6}{k{-}1}$& $\frac{6}{k}$ &
$\frac{6}{k{+}1}$ &
$\dots$&$\frac{6}{n{-}1}$ \\[-10pt]
&&&&&&&&&\\
\hline
&&&&&&&&&\\[-10pt]
$\begin{array}{c}\tilde V_{{-}1,{-}2{-}n}(n+1),\\  k=n \\\end{array}$ & ${-}\frac{6}{n{-}1}$& ${-}\frac{6}{n{-}2}$ & $\dots$ & ${-}\frac{6}{n{-}k{+}2}$& ${-}\frac{6}{n{-}k{+}1}$& $\frac{6}{n{-}k}$ &
$\frac{6}{n{-}k{-}1}$ &
$\dots$&${-}6$ \\[-10pt]
&&&&&&&&&\\
\hline
&&&&&&&&&\\[-10pt]
$\begin{array}{c}\tilde V_{{-}2,{-}3}(n+1),\\  k=1 \\\end{array}$ & $6$& $\frac{5}{2}$ & $\dots$ & $\frac{6(1{+}k)}{(k{-}1)k}$& $\frac{6(2{+}k)}{k(k{+}1)}$& $\frac{6(3{+}k)}{(k{+}1)(k{+}2)}$ &
$\frac{6(4{+}k)}{(k{+}2)(k{+}3)}$ &
$\dots$&$\frac{6(n{+}2)}{n(n{+}1)}$ \\[-10pt]
&&&&&&&&&\\
\hline
&&&&&&&&&\\[-10pt]
$\begin{array}{c} \tilde V_{1,{-}n}(n+1),\\  k=n \\\end{array}$ & ${-}\frac{6(n{+}2)}{n(n{+}1)}$& ${-}\frac{6(n{+}1)}{(n{-}1)n}$ & $\dots$ & $\dots$& ${-}\frac{6(n{+}4{-}k)}{(n{+}3{-}k)(n{+}2{-}k)}$& $\dots$ &${-}\frac{6(n{+}2{-}k)}{(n{+}1{-}k)(n{-}k)}$
 &
$\dots$&${-}6$ \\[-10pt]
&&&&&&&&&\\
\hline
\end{tabular}
\end{table}

4) Let $k \ge 5$. The equations on $b_{k{-}4}$ are
\begin{equation}
\begin{split}
R^5_{k{-}4}: \quad b_{k{-}1}(2b_{k{-}4}{-}3b_{k{-}3}{+}1)=0,\\
R^7_{k{-}4}:\quad
b_{k{-}3}(3b_{k{-}1}-3b_k)=b_{k{-}1}b_{k{+}1}{+}\frac{9}{2}b_k{-}9b_{k{-}1},
\end{split}
\end{equation}
they will give us for 1), 2), 3) respectively
$$
1) \; b_{k{-}4}=\frac{u{-}4}{2}, \quad 2) \;b_{k{-}4}={-}\frac{3}{2}, \quad 3) \; b_{k{-}4}={-}2.
$$

5) The remaining cases ($k \ge 6,7$) are treated similarly to the previous ones. 

\end{proof}
Now considering a general case of $(n+1)$-dimensional graded thread $W^+$-module $V$, $n+1 \ge 16$,
with $e_1f_k=0$ we may assume that $k+8 \le n$. If $k > n-8 \ge 7$  we take the dual $W^+$-module $V^*$ instead of $V$. We have for $V^*$
$$(\alpha_1^*,\alpha_2^*,\dots,\alpha_{n-1}^*, \alpha_{n}^*)=({-}\alpha_{n},{-}\alpha_{n-1},\dots, {-}\alpha_2, {-}\alpha_1).$$
It means that $\alpha^*_{n+1-k}=-\alpha_k=0$ and  $(n+1-k)+7 \le 15 \le n$.

Assuming $k+8 \le n$ we can define a
$9$-dimensional subquotient $\tilde V$ of $V$ 
$$
\tilde V =\langle
f_{k},\dots,f_{k+7},f_{k+8},\dots,f_n\rangle /\langle
f_{k+8},\dots,f_n\rangle
$$
and apply Lemmas \ref{odin_nol}, \ref{one_zero} and Corollary \ref{odin_nol_corollary}.

Every graded thread $W^+$-module $\tilde V$ (the set $(b_k,\dots,b_{k+7})$) presented 
in the classification lists of Lemmas \ref{odin_nol}, \ref{one_zero} and Corollary \ref{odin_nol_corollary} can be uniquely extended to 
the graded thread $W^+$-module $V$ ($(b_1,\dots,b_k,\dots,b_{k+8},\dots,b_{n-1})$. 
We remove the restriction $k+8 \le n$ considering their dual modules.

The results of this classification are presented in the Table \ref{type_odin_nul}.

\end{proof}

\section{Graded thread modules of the type $(1,\dots,1,0,0,1,\dots,1)$.}
Now we consider modules with vanishing two consecutive $\alpha_i, \alpha_{i+1}$, i.e. 
\begin{equation}
\label{reprise}
\exists ! \;k, 1\le k \le n-1,\;\alpha_{k}=\alpha_{k{+}1}=0.
\end{equation}
Modules of this type are called  "repris\'e" in \cite{Benoist}

In \cite{Mill} an infinite-dimensional graded thread $W^+$-module
$\tilde V_{gr}$ was constructed, it was defined by its basis $\{f_j, j \in {\mathbb
Z}\}$ and the relations
\begin{equation}
\label{new_module} e_if_j= \left\{\begin{array}{lll}
   jf_{i+j}, & j \ge 0; &\\
   (i+j)f_{i+j}, & i+j \le 0,&j<0;  \\
   f_{i+j},& i+j>0,& j<0.\\
   \end{array} \right .
\end{equation}
It holds  for this module
$
e_1f_{-1}=e_1f_0=0.
$
This module and its finite-dimensional subquotients played the crucial role in the proof of Buchstaber's conjecture on Massey products in Lie algebra cohomology $H^*(W^+, {\mathbb K})$ \cite{Mill}.  It has interesting nature, it is not a module of $V_{\lambda, \mu}$ family or its degeneration or deformation, in some sense  it is the result of "gluing together" of two modules: the quotient of $V_{{-}1,1}$ with a submodule of $V_{0,0}$ and it is unique infinite-dimensional module with the property $\exists ! \;k, \alpha_{k}=\alpha_{k{+}1}=0$ \cite{Mill} .

\begin{theorem}
Let $V$ be a $(n{+}1)$-dimensional, $n{+}1 \ge 11$, indecomposable graded thread
$W^+$-module of the type $(1,\dots,1,0,0,1,\dots,1)$, i.e. there exists a basis
$f_1,\dots,f_{n+1}$ of $V$ and $\exists! k, 1 \le k \le n{-}1$
such that:
\begin{eqnarray*}
\tilde e_1f_i=f_{i+1}, \;i=1,\dots,k{-}1,k{+}2,\dots,n{-}1;\\
\tilde e_1f_{k}=e_1f_{k{+}1}=0, \quad \tilde e_2f_j=b_jf_{j+2}, \quad
j=1,\dots,n{-}2.
\end{eqnarray*}
then if $k \ne 1, n{-}1,$ the module $V$ is isomorphic to one and only one module 
from the list
\begin{itemize}
\item
$R_k, 1 \le k \le n-1,$  defined by its basis $f_1,\dots, f_{n+1}$ and relations
\begin{equation}
\label{new_module} e_if_{j}= \left\{\begin{array}{ll}
   (j-k-1)f_{i+j}, & k+1 \le  j \le n+1, \; i+j \le n+1; \\
   (i+j-k-1)f_{i+j}, & i+j \le k+1, j<k+1;  \\
   f_{i+j},& k+1 <i+j\le n+1, j<k+1;\\
   0, & {\rm otherwise}.\\
   \end{array} \right .
\end{equation}
\item
its dual module $R_k^*, 1 \le k \le n-1$. 
\end{itemize}
\end{theorem}
\begin{proof}
The equations that involves $b_k, b_{k{+}1}, b_{k{+}2}$ will be
obtained from the standard ones by a very simple procedure: we will
keep the summands only of the form $b_kb_j$, $b_{k+1}b_l$,
$b_{k+2}b_l$.

1) Let consider the case $k=1$ (and hence the dual module
with $k=n{-}2$). The first four equations
$R^5_1,R^5_2,R^5_3,R^7_1$ are:
\begin{equation}
\label{dva_nulya_vnachale}
\begin{split}
2b_1 b_4{-} b_1 b_3{-}b_1=0,\\
2 b_2 b_5{-} b_2 b_4{-} b_2=0,\\
2 b_3 b_6{-} b_4 b_6{-} b_3b_5{-}(b_3{-}3b_4{+}3b_5{-}b_6)=0,\\
b_6b_1{-}b_1(b_3{-}3b_4{+}3b_5{-}b_6){-}\frac{9}{10} b_1=0.\\
\end{split}
\end{equation}
If $b_1=0$ then there is a decomposition $V=\langle f_1
\rangle\oplus \langle f_2,\dots,f_n \rangle$ in the sum of two
submodules, if $b_2=0$ then $V=\langle f_2 \rangle\oplus \langle
f_1,f_3,\dots,f_n \rangle$ is also the sum of its submodules.
Hence we may assume that $b_1=b_2=1$.

The system (\ref{dva_nulya_vnachale}) with $b_1=b_2=1$ has two
solutions
$$
(b_3,b_4,b_5,b_6)=\left(3,2,\frac{3}{2},\frac{6}{5}\right),\left(\frac{9}{5},\frac{7}{5},\frac{6}{5},\frac{21}{20}
\right).
$$
It follows from the Proposition \ref{odin_nol} that the module
$\langle f_2,f_3,\dots,f_7\rangle$ corresponding to the second
solution can not be extended to $\langle
f_2,f_3,\dots,f_7,\dots,f_n\rangle$ with $n \ge 10$. On the
another hand one can check out that we have the only one module
with 
$$b_1=1, b_2=1, b_3=3, b_4=2,\dots, b_{n{-}2}=\frac{6}{n{-}3}$$
that corresponds to the first solution.

2) Let us suppose now that $2 \le k \le n{-}6$.
 Then the equations
$R^5_{k{-}1},R^5_k,R^5_{k{+}1},R^7_{k{-}1}$ will have the
following form:
\begin{equation}
\label{dva_nulya}
\begin{split}
{-}b_{k} b_{k{+}2}{-} b_{k{-}1} b_{k{+}1}{+}3b_{k}=0,\\
2 b_{k} b_{k{+}3}{-} b_{k} b_{k{+}2}{-} b_{k}=0,\\
2 b_{k{+}1} b_{k{+}4}{-} b_{k{+}1} b_{k{+}3}{-} b_{k{+}1}=0,\\
{-}3 b_k b_{k{+}4}{-}b_{k{-}1} b_{k{+}1}{+}\frac{9}{2} b_k=0.\\
\end{split}
\end{equation}
\begin{proposition}
Let $b_k=0$ or $b_{k{-}1}= b_{k{+}1}=0$ then the module $V$ is
decomposable.
\end{proposition}
\begin{proof}
Indeed it follows from the first equation of the system above that
if $b_k=0$ then $ b_{k{-}1} b_{k{+}1}=0$. In the case of
$b_{k{-}1}=b_k=0$ we have the invariant decomposition:
$$
V=\langle f_1,\dots,f_k \rangle \oplus \langle f_{k{+}1},\dots,
f_{n}\rangle.
$$
If $ b_k=b_{k{+}1}=0$ then $V$ is decomposed in another way:
$$
V=\langle f_1,\dots,f_k, f_{k{+}1} \rangle \oplus \langle
f_{k{+}2},\dots, f_{n}\rangle.
$$
If both $b_{k{-}1}= b_{k{+}1}=0$ then $V$ is also decomposable:
$$
V=\langle f_1,\dots, f_k,  f_{k{+}2}, \dots, f_n\rangle \oplus
\langle f_{k{+}1}\rangle.
$$
\end{proof}
Now after some rescaling of the basic vectors we have to consider
two possibilities:
\begin{itemize}
\item[1)] $ b_{k{-}1}=b_k =1$;
\item[2)] $b_{k{-}1}=0,\;b_k= b_{k{+}1} =1$.
\end{itemize}

In the first case the system (\ref{dva_nulya}) has the only one
solution 
$$b_{k{+}1}{=}0, b_{k{+}2}{=}3, b_{k{+}3}{=}2,
b_{k{+}4}{=}\frac{3}{2}.$$ 

In the second case we also have the
unique solution $ b_{k{+}2}{=}3, b_{k{+}3}{=}2,
b_{k{+}4}{=}\frac{3}{2}$. The components $b_i, i > k{+}4$ are also
determined uniquely as it follows from the Proposition
\ref{odin_nol}. If we want to find $b_{k{-}2}$ we have to suppose
that $k \ge 3$ and to consider two new equations:
\begin{equation}
\begin{split}
b_{k{-}1} b_{k{+}1}{-} b_{k{-}2} b_k{-}3b_k=0,\\
3b_kb_{k{+}3}{-}b_{k{-}2}b_k{-}9b_k=0.\\
\end{split}
\end{equation}
Evidently in both cases ($b_{k{+}1}=0$ or $b_{k{+}1}=1$) we have
the same answer $b_{k{-}2}={-}3$.

Now supposing that $k \ge 4$ we have two new additional equations:
\begin{equation}
\begin{split}
2b_{k{-}3} b_k{-} b_{k{-}2} b_k{+}b_k=0,\\
{-}b_kb_{k{+}2}{+}3b_{k{-}3}b_k{+}9b_k=0.\\
\end{split}
\end{equation}
Again it follows that $b_{k{-}3}={-}2$ in both situations.

Let $k \ge 5$. We have two equations on $b_{k{-}4}$:
\begin{equation}
\begin{split}
2b_{k{-}4} b_{k{-}1}{-} b_{k{-}3} b_{k{-}1}{+}b_{k{-}1}=0,\\
{-}b_{k{+}1}b_{k{-}1}{-}3b_{k{-}4}b_k{-}\frac{9}{2}b_k=0.\\
\end{split}
\end{equation}
In both cases we have $b_{k{-}4}={-}\frac{3}{2}$.

The last case $k \ge 6$ can be treated absolutely similarly. In fact we
have obtained two graded modules:
$$
\begin{tabular}{|c|c|c|c|c|c|c|c|c|c|}
\hline
&&&&&&&&&\\[-10pt]
module & $b_{1}$&$\dots$&$b_{k{-}2}$ & $b_{k{-}1}$ & $b_k$
&
 $b_{k{+}1}$ & $b_{k{+}2}$ & $\dots$ &$b_{n{-}1}$  \\
&&&&&&&&&\\[-10pt]
\hline
&&&&&&&&&\\[-10pt]
 $R_k$ &${-}\frac{6}{k-1}$ & $\dots$ & ${-}3$& $1$& $1$ & $0$ &
$3$ &$\dots$&$\frac{6}{n{-}k{-}1}$ \\
&&&&&&&&&\\[-10pt]
\hline
&&&&&&&&&\\[-10pt]
 $R_{n{-}k{-}1}^*$ &${-}\frac{6}{k-1}$ & $\dots$ & ${-}3$& $0$& $1$ & $1$ &
$3$ &$\dots$&$\frac{6}{n{-}k{-}1}$ \\[-10pt]
&&&&&&&&&\\
\hline
\end{tabular}
$$
The cases $n{-}5 \le k \le n{-}3$ follow from the previous
considerations: one have to take the corresponding dual
module instead of initial one.

\end{proof}


\begin{thebibliography}{A}



\bibitem{BdFIZ}
Bauer, M., Di Francesco,  Ph., Itzykson, C.,  Zuber, J.-B.: Covariant differential equations and singular vectors in Virasoro representations. Nuclear Phys. B. \textbf{362}:3, 515--562 (1991)

\bibitem{ChP}
Chari, V., Pressley, A.: Unitary representations of the Virasoro algebra and a conjecture
of Kac. Compositio Mathematica \textbf{67}, 315--342 (1988)

\bibitem{Benoist}
Benoist, Y.:  Une nilvari\'et\'e  non affine. J. Differential
Geometry. \textbf{41}, 21--52 (1995)

\bibitem{FeFu}
Feigin, B., Fuchs, D.: Homology of the Lie algebras of vector
fields on the line. Funct. Anal. Appl. \textbf{14}:3, 45--60 (1980)
45--60.

\bibitem{FeFuRe}
Feigin, B.L., Fuchs, D.B., Retakh, V.S.: Massey operations in the cohomology of the infinite-dimensional Lie algebra $L_1$. In: Lecture Notes in Math, \textbf{1346}, pp.13--31. Springer-Verlag,
Zentralblatt Berlin (1988)

\bibitem{Fu}
Fuchs, D.: Cohomology of infinite-dimensional Lie algebras.
Consultants Bureau, N.Y., London (1986)

\bibitem{KacR}
Kac, V.G., Raina, A.K.: Highest weight representations of infinite dimensional Lie algebras
Adv. Ser. Math. Phys. 2, (1988)

\bibitem{Kenji_K}
Iohara, K., Y. Koga, Y.: Representation Theory of the Virasoro Algebra,  Springer Monographs in Math., Springer-Verlag, (2010)

\bibitem{KaSa}
Kaplansky, I., Santharoubane, L.J.: Harish Chandra modules over the Virasoro
algebra. Publ. Math. Sci. Res. Inst. \textbf{4}, 217--231 (1987)

\bibitem{MP}
Martin, C., Piard, A.: Indecomposable modules for the Virasoro Lie algebra and
a conjecture of Kac. Commun. Math. Phys. \textbf{137}, 109--132 (1991).

\bibitem{M}
Mathieu, O.: Classification of Harish-Chandra modules over the
Virasoro Lie algebra. Invent. Math. \textbf{107}, 225--234 (1992)

\bibitem{Mill}
Millionshchikov, D.: Algebra of formal vector fields on the line and Buchstaber's conjecture.
 Funct. Anal. Its Appl.  \textbf{43}:4, 264--278 (2009)

\bibitem{Mill2}
Millionshchikov, D.: Virasoro singular vectors.
 Funct. Anal. Its Appl. \textbf{50}:3, 219--224 (2016)

\bibitem{Milnor}
Milnor, J.: On fundamental groups of complete affinely flat manifolds.  Adv. Math.,  \textbf{25}, 178--187 (1977)
\end{thebibliography}
\end{document}